\documentclass[10pt]{amsart}
\usepackage{latexsym, amssymb, amsfonts, amscd}

\theoremstyle{plain}
\newtheorem{Thm}{Theorem}[section]
\newtheorem{Thm*}{Theorem}[section]

\newtheorem{Prop}[Thm]{Proposition}
\newtheorem{Lem}[Thm]{Lemma}
\newtheorem{Cl}[Thm]{Claim}

\theoremstyle{definition}

\newtheorem{Emp}[Thm]{}

\newcommand{\sotimes}{\overset{\centerdot}{\otimes}}
\newcommand{\fa}{\mathfrak a}
\numberwithin{equation}{section}

\newcommand{\qlbar}{\overline{\B{Q}}_{\ell}}

\newcommand{\om}{\omega}

\newcommand{\La}{\Lambda}

\newcommand{\ov}{\overline}

\newcommand{\B}[1]{\mathbb#1}
\newcommand{\cal}[1]{\mathcal{#1}}
\newcommand{\C}[1]{\cal#1}

\newcommand{\isom}{\overset {\thicksim}{\to}}

\newcommand{\form}[1]{(\ref{Eq:#1})}
\newcommand{\ka}{\kappa}

\newcommand{\st}{\operatorname{st}}
\newcommand{\Ind}{\operatorname{Ind}}
\newcommand{\Bun}{\operatorname{Bun}}
\newcommand{\Perf}{\operatorname{Perf}}
\newcommand{\FS}{\operatorname{FS}}
\newcommand{\lis}{\operatorname{lis}}

\newcommand{\si}{\sigma}

\newcommand{\lra}{\longrightarrow}

\newcommand{\hra}{\hookrightarrow}
\newcommand{\wt}{\widetilde}
\newcommand{\wh}{\widehat}
\newcommand{\Gm}{\Gamma}

\newcommand{\gm}{\gamma}

\newcommand{\Dt}{\Delta}

\newcommand{\lan}{\langle}
\newcommand{\ran}{\rangle}

\newcommand{\al}{\alpha}

\newcommand{\la}{\lambda}

\newcommand{\rl}[1]{Lemma~\ref{L:#1}}

\newcommand{\rp}[1]{Proposition~\ref{P:#1}}

\newcommand{\re}[1]{\ref{E:#1}}

\newcommand{\rt}[1] {Theorem~\ref{T:#1}}

\newcommand{\Out}{\operatorname{Out}}
\newcommand{\sr}{\operatorname{sr}}
\newcommand{\on}{\operatorname}

\newcommand{\inv}{\operatorname{inv}}
\newcommand{\pr}{\operatorname{pr}}

\newcommand{\Ker}{\operatorname{Ker}}

\newcommand{\Irr}{\operatorname{Irr}}

\newcommand{\Spec}{\operatorname{Spec}}
\newcommand{\Aut}{\operatorname{Aut}}

\newcommand{\res}{\operatorname{res}}
\newcommand{\Ad}{\operatorname{Ad}}

\newcommand{\Gal}{\operatorname{Gal}}
\newcommand{\Tr}{\operatorname{Tr}}

\newcommand{\Lie}{\operatorname{Lie}}

\newcommand{\ad}{\operatorname{ad}}

\newcommand{\Id}{\operatorname{Id}}
\newcommand{\Rep}{\operatorname{Rep}}
\newcommand{\Par}{\operatorname{Par}}

\newcommand{\End}{\operatorname{End}}
\newcommand{\QCoh}{\operatorname{QCoh}}

\newcommand{\pt}{\operatorname{pt}}

\newcommand{\Span}{\operatorname{Span}}

\usepackage{xcolor}
\usepackage{amscd}
\begin{document}

\title[Endoscopic decomposition of elliptic Fargues--Scholze $L$-packets]%
{Endoscopic decomposition of elliptic Fargues--Scholze $L$-packets}
\author{David Kazhdan}

\author{Yakov Varshavsky}
\address{Einstein Institute of Mathematics\\
Edmond J. Safra Campus\\
The Hebrew University of Jerusalem\\
Givat Ram, Jerusalem, 9190401, Israel}
\email{kazhdan@math.huji.ac.il, yakov.varshavsky@mail.huji.ac.il}

\date{\today}

\thanks{The research was partially supported by ERC grant 101142781. The research of Y.V. was partially supported by the
ISF grant 2091/21.}
\begin{abstract}
The main goal of this note is to show that the local $L$-packets of Fargues--Scholze
\cite{FS}, corresponding to elliptic $L$-parameters, admit an endoscopic decomposition.
Our argument is strongly motivated by a beautiful paper of Chenji Fu \cite{Fu}, where
the stable case is proven.

To put our results in a more general context, we also construct a general endoscopic decomposition over $\B{C}$
based on results of Arthur, and a generalization of this decomposition over an arbitrary
algebraically closed field of characteristic zero based on a recent work \cite{KSV}.
\end{abstract}
\maketitle


\tableofcontents

\section*{Formulation of results}

\begin{Emp}
{\bf Notation.}
(a) Let $F$ be a local non-archimedean field of characteristic zero, let $\Gm:=\Gm_F$ be the absolute Galois group of $F$, let $W_F\subseteq\Gm$ be the Weil group of $F$, and let $\ell$ a prime different from the residue characteristic of $F$. Untill Section~\ref{ell}, all functions will be $\qlbar$-valued.

\smallskip

(b) Let $G$ be a connected reductive group over $F$, and let $\wh{G}$ the connected Langlands dual group over $\qlbar$. Then $\wh{G}$ is equipped with an action of $\Gm$ (canonical up to a conjugacy), so we can form the semidirect product ${}^L G:=\wh{G}\rtimes W_F$.

\smallskip

(c) Let $G^{\on{sr}}\subseteq G$ be the (open) locus of strongly regular semisimple elements of $G$, and let $G(F)^{\on{ell}}\subseteq G^{\on{sr}}(F)$ be the (open) locus of elliptic elements.

\smallskip

(d) Let $C^{\infty}_{\inv}(G^{\on{sr}}(F))$ and $C^{\infty}_{\inv}(G(F)^{\on{ell}})$ be the spaces of $\Ad G(F)$-invariant locally constant functions on $G^{\on{sr}}(F)$ and $G(F)^{\on{ell}}$, respectively, and let $\wh{C}_{\inv}(G(F))$ be the space of $\Ad G(F)$-invariant generalized functions on $G(F)$.
\end{Emp}

\begin{Emp} \label{E:smooth}
{\bf Smooth irreducible representations.} (a) We denote by $\Irr(G(F))$ the set of isomorphism classes of smooth irreducible representations $G(F)$ over $\qlbar$.

\smallskip

(b) For every $\pi\in \Irr(G(F))$, we denote its character by $\chi_{\pi}\in \wh{C}_{\inv}(G(F))$, and we denote by
$\Theta(G(F))\subseteq \wh{C}_{\inv}(G(F))$ the $\qlbar$-linear span of $\{\chi_{\pi}\}_{\pi\in \Irr(G(F))}$.

\smallskip

(c) By a theorem of Harish--Chandra, the restriction $f^{\on{sr}}:=f|_{G^{\on{sr}}(F)}$ to $G(F)^{\on{sr}}$ of every $f\in \Theta(G(F))$ is a locally constant function, and thus defines an element of ${C}^{\infty}_{\inv}(G^{\on{sr}}(F))$. We denote by $f^{\on{ell}}\in C^{\infty}_{\inv}(G(F)^{\on{ell}})$ the restriction of $f$ to $G(F)^{\on{ell}}$.
\end{Emp}

\begin{Emp} \label{E:lpackets}
{\bf Fargues--Scholze $L$-packets.}
(a) To every $\pi\in\Irr(G(F))$, Fargues and Scholze \cite{FS} associate a semisimple Langlands parameter
$\varphi_{\pi}:W_F\to {}^L G(\qlbar)$.

\smallskip

(b) Conversely, to every semisimple $L$-parameter $\varphi$ we associate a subset
\[
\Pi^{FS}_{\varphi}:=\{\pi\in \Irr(G(F))\,|\, \varphi^{FS}_{\pi}=\pi\}
\]
and call it the Fargues--Scholze $L$-packet.

\smallskip

(c) For every $L$-parameter $\varphi$ as in part~(b), we denote by $\Theta_{\varphi}\subseteq\Theta(G(F))$ the span of $\{\chi_{\pi}\}_{\pi\in\Pi^{FS}_{\varphi}}$. 

\end{Emp}

\begin{Emp} \label{E:elliptic}
{\bf Elliptic Langlands parameters.}
(a) For every (semisimple) $L$-parameters $\varphi$ (see Section~\re{lpackets}(a)), we denote by $S_{\varphi}:=Z_{\wh{G}}(\varphi)\subseteq \wh{G}(\qlbar)$ the centralizer of $\varphi$. Then $S_{\varphi}\supseteq Z(\wh{G})^{\Gm}$, and we set
$\ov{S}_{\varphi}:=S_{\varphi}/Z(\wh{G})^{\Gm}\subseteq \wh{G}^{\ad}$.

\smallskip

(b) Assume now that the $L$-parameter $\varphi$ is {\em elliptic}, by which we mean that the quotient
$\ov{S}_{\varphi}:=S_{\varphi}/Z(\wh{G})$ is finite.

\smallskip

(c) Following Langlands \cite{La}, every $\ov{S}_{\varphi}$-conjugacy class $[\ka]\subseteq \ov{S}_{\varphi}$ gives rise to the endoscopic datum $\C{E}_{\varphi,[\ka]}$ of $G$, unique up to an equivalence (see Section~\re{enddatum}). Thus, we can talk about  $\C{E}_{\varphi,[\ka]}$-stable generalized functions (see Section~\re{estable}).
\end{Emp}

\begin{Emp} \label{E:spact}
{\bf The canonical decomposition.} (a) Let $\qlbar[\ov{S}_{\varphi}]^{\ov{S}_{\varphi}}$ be the algebra of $\qlbar$-valued functions on $\ov{S}_{\varphi}$, invariant under conjugation. The spectral action of Fargues--Scholze induces an action of the $\qlbar$-algebra $\qlbar[\ov{S}_{\varphi}]^{\ov{S}_{\varphi}}$ on the $\qlbar$-vector space $\Theta_{\varphi}$ (see Section~\re{main}).

\smallskip

(b) For every $\ov{S}_{\varphi}$-conjugacy class $[\ka]\subseteq \ov{S}_{\varphi}$, we denote by $1_{[\ka]}\in \qlbar[\ov{S}_{\varphi}]^{\ov{S}_{\varphi}}$ the characteristic function of $[\ka]$. Then we have an identity
$\sum_{[\ka]}1_{[\ka]}=1\in \qlbar[\ov{S}_{\varphi}]^{\ov{S}_{\varphi}}$.

\smallskip

(c) Combining observations of parts~(a) and (b), for every $f\in \Theta_{\varphi}$ we have a natural decomposition $f=\sum_{[\ka]}f_{[\ka]}$, where for every $\ov{S}_{\varphi}$-conjugacy class $[\ka]\subseteq \ov{S}_{\varphi}$, we set $f_{[\ka]}:=1_{[\ka]}(f)\in\Theta_{\varphi}$.
\end{Emp}

Now we are ready to formulate the main result of this note.

\begin{Thm} \label{T:main}

For every elliptic $L$-parameter $\varphi$, every $f\in \Theta_{\varphi}$ and every $[\ka]\subseteq \ov{S}_{\varphi}$, the generalized function
$f_{[\ka]}\in \Theta_{\varphi}$ is $\C{E}_{\varphi,[\ka]}$-stable.
\end{Thm}

\begin{Emp}
{\bf Remark.} Both \rt{main} and its proof were motivated by a recent work of Chenji Fu \cite{Fu}, who established  it in the case $[\ka]=1$. 
On the other hand, even in the case $[\ka]=1$ our argument is slightly different. Namely, instead of equi-distribution properties of the weight multiplicities of highest weight representations of a reductive group we use the fact that the space of class functions
is generated by traces of irreducible representations. 

%
%
%
\end{Emp}

\begin{Emp}
{\bf Outline of the argument.} Our proof of \rt{main} consists of three steps. First, we use (a reformulation of) the character formula of Hansen--Kaletha--Weinstein \cite{HKW} to show the $\C{E}_{\varphi,[\ka]}$-stability  of the restriction $f^{\on{ell}}_{[\ka]}:=(f_{[\ka]})|_{G(F)^{\on{ell}}}$ when the center $Z(G)$ is connected. Then, we use $z$-embeddings of Kaletha \cite{Ka} to establish the $\C{E}_{\varphi,[\ka]}$-stability  of $f^{\on{ell}}_{[\ka]}$ in general. Finally, we use the results of Arthur \cite{Ar1, Ar2} in the formulation of Moeglin--Waldspurger \cite{MW} in order to deduce the $\C{E}_{\varphi,[\ka]}$-stability of $f_{[\ka]}$ from the $\C{E}_{\varphi,[\ka]}$-stability of $f^{\on{ell}}_{[\ka]}$.
\end{Emp}

The decomposition in \rt{main} is a particular case of the following more general phenomenon:

\begin{Emp} \label{E:enddec0}
{\bf General endoscopic decomposition.} (a) Using results of Arthur, for every field isomorphism $\iota:\qlbar\isom\B{C}$, the space $\Theta(G(F))$ has a natural decomposition
\[
\Theta(G(F))=\bigoplus_{\C{E}}\Theta_{\C{E}}(G(F)),
\]
where $\C{E}$ runs over equivalences classes of endoscopic data (see Appendix~\ref{end}).

\smallskip

(b) More precisely, the decomposition of part~(a) is obtained from a combination of a decomposition of $\Theta(G(F))$ in terms of
virtual elliptic characters in the sense of Arthur for different Levi subgroups of $G$, and a decomposition of the span of virtual elliptic characters in terms of elliptic endoscopic data.

\smallskip

(c) (A version of) Langlands conjecture asserts that for every semisimple $L$-parameter $\varphi$ and every field isomorphism $\iota$, the subspace $\Theta_{\varphi}$ of Section~\re{lpackets} is {\em compatible with the decomposition} of part~(a), that is, we have a decomposition
\[
\Theta_{\varphi}=\bigoplus_{\C{E}}(\Theta_{\C{E}}(G(F))\cap \Theta_{\varphi}).
\]

\smallskip

(d) \rt{main} establishes (a more precise version of) the conjecture of part~(c) for elliptic $L$-parameters.
\end{Emp}

\begin{Emp} \label{E:indep}
{\bf Independence of $\iota$.} (a) A natural question is whether the decomposition of Section~\re{enddec0}(a) is independent of the field isomorphism $\iota$. More generally, one can ask whether this decomposition can be defined over an arbitrary algebraically closed field of zero characteristic. 

\smallskip

(b) In Appendix~\ref{S:arb} we give a positive answer to the question of part~(a) based on a recent work \cite{KSV}. This can be considered to be a counterpart of the recent result of Scholze \cite{Sc}, asserting that the Fargues--Scholze correspondence can be done motivically and thus is independent of $\ell$ and $\iota$.
\end{Emp}

\begin{Emp}
{\bf Plan of the paper.} The paper is organized as follows. First, in Section~1, we review the definition of $\C{E}$-stability. After that,
in Section~2, we give a proof of \rt{main} assuming Propositions~\ref{P:HKW},
\ref{P:zemb} and \ref{P:ell}. Next, in Section~3, we review basic properties of the spectral action of \cite{FS} and use them to construct the
action of the algebra $\qlbar[\ov{S}_{\varphi}]^{\ov{S}_{\varphi}}$ on $\Theta_{\varphi}$, described in Section~\re{spact}. Then, in Sections~4-6, we provide proofs of Propositions~\ref{P:HKW}, \ref{P:zemb} and \ref{P:ell}, respectively.

Finally, in two appendices we provide a construction of the endoscopic decomposition over $\B{C}$ mentioned above, and of the decomposition over an arbitrary algebraically closed field of characteristic zero, respectively.
\end{Emp}

\begin{Emp}
{\bf Acknowledgements.} Our work own its existence to a beautiful work of Chenji Fu \cite{Fu}, in which the stable case is proven. We thank David Hansen, Peter Scholze and Sandeep Varma for their interest and helpful conversations. We also thank Jean--Loup Waldspurger for his explanations and references, and Chenji Fu for his comments on the first draft. Finally, we thank Maarten Solleveld for his collaboration on the paper \cite{KSV} on which Appendix~~\ref{S:arb} is based. Part of the work has been carried out when the second author visited the MPIM Bonn.
\end{Emp}

\section{$\C{E}$-stability}

\begin{Emp} \label{E:dualgp}
{\bf Langlands dual groups.}

\smallskip

(a) For every connected reductive group $H$ over $F$, we denote by $\wh{H}$  \label{a:ghat} the (connected) Langlands dual group over $\qlbar$. Then $\wh{H}$ is equipped with a canonical continuous homomorphism $\rho_H:\Gm\to \on{Out}(\wh{H})$. In particular, we have a natural action of
$\Gm$ on $Z(\wh{H})$, so we can consider the group of fixed points $Z(\wh{H})^{\Gm}$. Moreover, $\rho_H$ lifts to an action of $\Gm$ on $\wh{H}$, canonically up to an $\wh{H}$-conjugacy.

\smallskip

(b) Every embedding  $\fa:T\hra H$ of a maximal torus over $F$ gives rise to a $\Gm$-invariant $\wh{H}$-conjugacy class of embeddings
$[\wh{\fa}]:\wh{T}\hra \wh{H}$ of maximal tori, hence to a canonical $\Gm$-equivariant embedding $Z(\wh{H})\hra\wh{T}$.

\end{Emp}

\begin{Emp} \label{E:bor}
{\bf Galois cohomology.}

\smallskip

(a) As it was shown by Borovoi \cite{Bo}, for every connected reductive group $H$ over $F$, the Galois cohomology
$H^1(F,H)$ has the natural structure of a finite abelian group isomorphic to $\pi_1(H)_{\Gm,\on{tor}}$, where $\pi_1$ stands for the algebraic fundamendal group, $(-)_{\Gm}$ stands for coinvariants, and $(-)_{\on{tor}}$ stands for torsion. Furthermore, the isomorphism
$H^1(F,H)\isom\pi_1(H)_{\Gm,\on{tor}}$ is functorial in $H$.

\smallskip

(b) Part~(a) implies that we have a natural isomorphism $H^1(F,H)^{\vee}\simeq \pi_0(Z(\wh{H})^{\Gm})$, where $(-)^{\vee}$ stands for the dual abelian group. Moreover, for every maximal torus $T\subseteq H$, we have a canonical isomorphism
\[
\Ker[H^1(F,T)\to H^1(F,H)]^{\vee}\simeq \pi_0(\wh{T}^{\Gm}/Z(\wh{H})^{\Gm}),
\]
first established by Kottwitz \cite{Ko}.
\end{Emp}

\begin{Emp} \label{E:stconj}
{\bf Stable conjugacy class.} Fix $\gm\in G^{\on{sr}}(F)$.

\smallskip

(a) We denote by $\on{Orb}^{\st}_{\gm}:=\on{Orb}^{\on{st}}_{G(F)}(\gm)\subseteq G(F)^{\on{ell}}$ the {\em stable  conjugacy class} of $\gm$,
that is, the set of all $\gm'\in G(F)$ which are $G(\ov{F})$-conjugate to $\gm$, and let $[\on{Orb}^{\st}_{\gm}]$ be the set of $G(F)$-orbits in $\on{Orb}^{\st}_{\gm}$.

\smallskip

(b) Let $G_{\gm}\subseteq G$ be the centralizer of $\gm$, and consider a finite abelian group
\[
A_{\gm} :=\Ker[H^1(F, G_{\gm})\to H^1(F, G)].
\]
Then $[\on{Orb}^{\st}_{\gm}]$ has a natural stucture of an $A_{\gm}$-torsor. Namely, for $\ov{\gm}',\ov{\gm}''\in [\on{Orb}^{\st}_{\gm}]$ one can form
a relative position $\inv(\ov{\gm}',\ov{\gm}'')\in A_{\gm}$ (compare \cite[Section 4.1.2]{BV}), and the action of $A_{\gm}$ on $[\on{Orb}^{\st}_{\gm}]$ is defined by the rule $\inv(\ov{\gm}',\ov{\gm}'')(\ov{\gm}')=\ov{\gm}''$.

\smallskip

(c) Let $\wh{G}_{\gm}$ be the connected Langlands dual group of $G_{\gm}$. Then, by Section~\re{bor}(b), the dual abelian group $A_{\gm}^{\vee}$ is canonically isomorphic to $\pi_0((\wh{G}_{\gm})^{\Gm}/Z(\wh{G})^{\Gm})$. In particular, every element $\xi\in (\wh{G}_{\gm})^{\Gm}/Z(\wh{G})^{\Gm}$ defines an element $\ov{\xi}\in A_{\gm}^{\vee}$.

\end{Emp}

\begin{Emp} \label{E:notstconj}
{\bf Notation.} Fix $\gm\in G^{\on{sr}}(F)$.

\smallskip

(a) Let $\on{Fun}([\on{Orb}^{\st}_{\gm}])$ be the space of $\qlbar$-valued functions on $[\on{Orb}^{\st}_{\gm}]$.
Then the finite abelian group $A_{\gm}$ has a natural action on $\on{Fun}([\on{Orb}^{\st}_{\gm}])$ such that for every $\la\in A_{\gm}$, $\ov{\gm}'\in[\on{Orb}^{\st}_{\gm}]$  and $f\in \on{Fun}([\on{Orb}^{\st}_{\gm}])$ we have
$\la(f)(\ov{\gm}')=f(\la(\ov{\gm}'))$.

\smallskip

(b) For $\ov{\xi}\in A_{\gm}^{\vee}$, let $\on{Fun}([\on{Orb}^{\st}_{\gm}])_{\ov{\xi}}\subseteq \on{Fun}([\on{Orb}^{\st}_{\gm}])$ be the subspace of $\ov{\xi}$-eigenvectors of $A_{\gm}$. Then each $\on{Fun}([\on{Orb}^{\st}_{\gm}])_{\ov{\xi}}$ is a one dimensional $\qlbar$-vector space, and
we have a decomposition $\on{Fun}([\on{Orb}^{\st}_{\gm}])=\bigoplus_{\ov{\xi}\in A_{\gm}^{\vee}}\on{Fun}([\on{Orb}^{\st}_{\gm}])_{\ov{\xi}}$.

\smallskip

(c) For a semisimple conjugacy class $[s]\subseteq\wh{G}^{\ad}$, let $\on{Fun}([\on{Orb}^{\st}_{\gm}])_{[s]}\subseteq \on{Fun}([\on{Orb}^{\st}_{\gm}])$ be the span of $\on{Fun}([\on{Orb}^{\st}_{\gm}])_{\ov{\xi}}$, where $\xi$ runs over $[(\wh{G}_{\gm})^{\Gm}/Z(\wh{G})^{\Gm}]\cap [s]\subseteq\wh{G}^{\ad}$, and, as in Section~\re{stconj}(c), $\ov{\xi}\in A_{\gm}^{\vee}$ denotes the
class of $\xi\in (\wh{G}_{\gm})^{\Gm}/Z(\wh{G})^{\Gm}$.

\smallskip

(d) For every function $f\in C_{\inv}(G^{\sr}(F))$, its restriction
$\res_{\gm}(f):=f|_{\on{Orb}^{\st}_{\gm}}$ is a $G(F)$-invariant function on $\on{Orb}^{\st}_{\gm}$, and thus defines an element of
$\on{Fun}([\on{Orb}^{\st}_{\gm}])$.

\end{Emp}

\begin{Emp} \label{E:enddatum}
{\bf Endoscopic data} (compare \cite[$\S$7]{Ko} and \cite[Section~1.3]{KV}).
 Suppose we are in the situation of Section~\re{dualgp}.

\smallskip

(a) Slightly modifying the terminology of \cite[$\S$7]{Ko}, by an {\em endoscopic datum}  \label{a:enddatum} for $G$, we mean a triple $\C{E}=(s,\wh{H},\rho)$, where $s\in\wh{G}^{\ad}$ is a semisimple element, $\wh{H}:=Z_{\wh{G}}(s)^0\subseteq \wh{G}$ the connected centralizer,  and $\rho:\Gm\to\on{Out}(\wh{H})$ a continuous homomorphism satisfying the following conditions:

\smallskip

\quad\quad$\bullet$ the $\wh{G}$-conjugacy class of the inclusion $\wh{H}\hra\wh{G}$ is $\Gm$-invariant, and thus the inclusion $Z(\wh{G})\hra Z(\wh{H})$ is $\Gm$-invariant;

\smallskip

\quad\quad$\bullet$ we have $s\in Z(\wh{H})^{\Gm}/Z(\wh{G})^{\Gm}\subseteq\wh{G}^{\ad}$.

\smallskip

(b) We call two endoscopic data $\C{E}=(s,\wh{H},\rho)$ and $\C{E}'=(s',\wh{H}',\rho')$ {\em equivalent} and write $\C{E}\simeq\C{E}'$, if there exists $g\in\wh{G}$ such that $\Ad_g(\wh{H})=\wh{H}'$, $\Ad_g\circ\rho=\rho'$ and elements
$\Ad_g(s),s'\in Z(\wh{H}')^{\Gm}$ have the same image in $\pi_0(Z(\wh{H}')^{\Gm}/Z(\wh{G})^{\Gm})$. 

\smallskip

(c) We call an endoscopic datum $\C{E}=(s,\wh{H},\rho)$ {\em elliptic}, if the group $Z(\wh{H})^{\Gm}/Z(\wh{G})^{\Gm}$ is finite.
\end{Emp}

\begin{Emp} \label{E:langl}
{\bf Endoscopic data corresponding to Langlands parameters.} Let $\varphi:W_F\to {}^LG(\qlbar)$ be an $L$-parameter, and let $\ka\in \ov{S}_{\varphi}$ be a semisimple element.

\smallskip

(a) As in \cite{La}, a pair $(\varphi,\ka)$ gives rise to the endoscopic datum $\C{E}_{\varphi,\ka}=(s,\wh{H},\rho)$, where $s:=\ka$ and $\rho:\Gm\to \Out(\wh{H})$ is characterised by the condition that the restriction $\rho|_{W_F}$ is the composition
\begin{equation} \label{Eq:endoscopy}
W_F\overset{\varphi}{\lra}Z_{{}^LG}(s)\overset{\Ad}{\lra}\Aut(\wh{H})\overset{\pr}{\lra}\Out(\wh{H}).
\end{equation}
Namely, it is easy to see that the composition \form{endoscopy} is continuous with a finite image, so it uniquely extends to the whole of $\Gm$.

\smallskip

(b) By the construction of part~(a), we have an equality $Z(\wh{H})^{\Gm}=Z_{Z(\wh{H})}(\varphi)$. Therefore we have an inclusion
$Z(\wh{H})^{\Gm}\subseteq Z_{\wh{G}}(\varphi)=S_{\varphi}$. In particular, for every elliptic $L$-parameter $\varphi$, the endoscopic datum $\C{E}_{\varphi,\ka}$ is elliptic.


\end{Emp}

\begin{Emp} \label{E:estable}
{\bf $\C{E}$-stable functions.} Let $\C{E}=(s,\wh{H},\rho)$ be an endoscopic datum for $G$.

\smallskip

(a) Let $\gm\in G^{\sr}(F)$, and let $\fa_{\gm}:G_{\gm}\hra G$ be the inclusion. For an element $\xi\in (\wh{G}_{\gm})^{\Gm}/Z(\wh{G})^{\Gm}$, we say that $(\gm,\xi)\in\C{E}$ if  there exists an embedding $\wh{\fa_{\gm}}:\wh{G}_{\gm}\hra \wh{G}$ in $[\wh{\fa_{\gm}}]$ (see Section~\re{dualgp}(b)) such that

\smallskip

\quad\quad$\bullet$  we have $\wh{\fa_{\gm}}(\xi)=s$, thus  $\wh{\fa_{\gm}}(\wh{G}_{\gm})\subseteq\wh{H}$;

\smallskip

\quad\quad$\bullet$ the $\wh{H}$-conjugacy class of the inclusion $\wh{\fa_{\gm}}:\wh{G}_{\gm}\hra\wh{H}$ is $\Gm$-invariant.

\smallskip

(b) We say that a function $f\in\on{Fun}([\on{Orb}^{\st}_{\gm}])$ is {\em $\C{E}$-stable}, if it belongs to the span of the subspaces
$\on{Fun}([\on{Orb}^{\st}_{\gm}])_{\ov{\xi}}$, where $\xi$ runs over all elements of $(\wh{G}_{\gm})^{\Gm}/Z(\wh{G})^{\Gm}$ such that $(\gm,\xi)\in\C{E}$.

\smallskip

(c) We say that a function $f\in C^{\infty}_{\inv}(G^{\sr}(F))$ (resp. $f\in C^{\infty}_{\inv}(G^{\on{ell}}(F))$) is {\em $\C{E}$-stable}, if its restriction $f|_{\on{Orb}^{\st}_{\gm}}\in\on{Fun}([\on{Orb}^{\st}_{\gm}])$ is $\C{E}$-stable for every $\gm\in G^{\sr}(F)$
(resp. $f\in C^{\infty}_{\inv}(G^{\on{ell}}(F))$).  We denote the space of $\C{E}$-stable invariant functions by
$C^{\infty}_{\C{E}-\st}(G(F)^{\on{ell}})\subset C^{\infty}_{\inv}(G(F)^{\on{ell}})$.

\smallskip

(d) Note that a function $f\in C^{\infty}_{\inv}(G^{\sr}(F))$, whose support is closed in $G(F)$, defines a generalized function on $G(F)$.
We say that an element $h$ in the Hecke algebra $\C{H}(G(F))$ is {\em $\C{E}$-unstable}, if for every $f\in C^{\infty}_{\inv}(G^{\sr}(F))$, whose support is closed in $G(F)$, we have $\lan f, h\ran=0$.

\smallskip

(e) We say that a generalized function  $f\in \wh{C}_{\inv}(G(F))$ is {\em $\C{E}$-stable}, if its satisfies $\lan f, h\ran=0$ for every $\C{E}$-unstable $h\in\C{H}(G(F))$. We denote the space of $\C{E}$-stable invariant generalized functions by $\wh{C}_{\C{E}-\st}(G(F))\subseteq\wh{C}_{\inv}(G(F))$.
\end{Emp}

\begin{Emp} \label{E:stable}
{\bf Remarks.} (a) $\C{E}$-stability only depends on the equivalence class of $\C{E}$.

\smallskip

(b) For every connected reductive group $G$, we have a trivial endoscopic datum $\C{E}_{G,\on{triv}}=\C{E}_{\on{triv}}:=(1,\wh{G},\rho_G)$. Then a function  $f\in C^{\infty}_{\inv}(G^{\sr}(F))$ is $\C{E}_{\on{triv}}$-stable if and only if
its restriction to each stable orbit $\on{Orb}^{\st}_{\gm}$ is constant, and such functions are usually called {\em stable}.  \label{a:stable}

\smallskip

(c) Note that if  $f\in C^{\infty}_{\inv}(G^{\sr}(F))$ is $\C{E}$-stable, then its restriction $f|_{\on{Orb}^{\st}_{\gm}}$ belongs to $\on{Fun}([\on{Orb}^{\st}_{\gm}])_{[s]}$ for every $\gm\in G^{\sr}(F)$. Furthermore, the converse assertion also holds, if the center $Z(G)$ of $G$ is connected (see \cite[Section~4.1.9]{BV}).

\smallskip

(d) For every pair $(\gm,\xi)$ as in Section~\re{estable}(a) there exists a unique endoscopic datum
$\C{E}=(s,\wh{H},\rho)$ such that $(\gm,\xi)\in\C{E}$ (compare \cite[Sections~4.1.7(a),(b)]{BV}). Furthermore,
$\C{E}$ is elliptic, if $\gm\in G(F)$ is elliptic.

\smallskip

(e) In words, an invariant generalized function is $\C{E}$-stable, if it lies in the closure of the span of all $\C{E}$-stable functions on
$G^{\sr}(F)$ having closed support in $G(F)$. In particular, we do not have to choose Haar measures in order to define this notion.

\smallskip

(f) Following \cite[Section~1.6.4]{KV}, for every $\gm\in G^{\sr}(F)$ we choose a Haar measure on $G_{\gm}(F)$ such that the total measure of the
maximal compact subgroup of $G_{\gm}(F)$ is one, and define an orbital integral $O_{\gm}\in\wh{C}_{\inv}(G(F))$ and a $\ka$-orbital integral
$O^{\ov{\xi}}_{\gm}\in\wh{C}_{\inv}(G(F))$ for every $\xi\in\pi_0((\wh{G}_{\gm})^{\Gm}/Z(\wh{G})^{\Gm})$. Then one can show that an invariant generalized function is $\C{E}$-stable if and only if it lies in the closure of the span of $\ka$-orbital integrals $\{O^{\ov{\xi}}_{\gm}\}_{(\gm,\xi)\in\C{E}}$. In particular, our definition coincides with the classical one.
\end{Emp}

As it will be explained in Section~\ref{ell}, the following result is an easy consequence of results of \cite{Ar1, Ar2} and \cite{MW}.

\begin{Prop} \label{P:ell}
Let $\pi_1,\ldots,\pi_n\in \Irr(G(F))$ be supercuspidal representations, let $\C{E}$ be an elliptic endoscopic datum for $G$,  and let $f=\sum_{i=1}^n a_i \chi_{\pi_i}\in\Theta(G(F))$ be a linear combination of the $\chi_{\pi_i}$'s such that the restriction
$f^{\on{ell}}\in C^{\infty}_{\inv}(G(F)^{\on{ell}})$ is $\C{E}$-stable. Then $f$ is $\C{E}$-stable.
\end{Prop}







\section{Proof of the Main Theorem}

\begin{Emp} \label{E:chev}
{\bf The Chevalley space.}

\smallskip

(a) Let $c_{\wh{G}^{\ad}}:=\Spec \qlbar[\wh{G}^{\ad}]^{\wh{G}^{\ad}}$ be the Chevalley space of $\wh{G}^{\ad}$, where $\qlbar[\wh{G}^{\ad}]^{\wh{G}^{\ad}}$ is the algebra of class functions, and let $\pr_{\wh{G}^{\ad}}:\wh{G}^{\ad}\to c_{\wh{G}^{\ad}}$ be the projection, corresponding to the inclusion $\qlbar[\wh{G}^{\ad}]^{\wh{G}^{\ad}}\hra \qlbar[\wh{G}^{\ad}]$.

\smallskip

(b) For every semisimple conjugacy class $[s]\subseteq\wh{G}^{\ad}$, the map $h\mapsto h([s])$ defines a character
$\zeta_{[s]}:\qlbar[c_{\wh{G}^{\ad}}]=\qlbar[\wh{G}^{\ad}]^{\wh{G}^{\ad}}\to\qlbar$. Conversely, every character $\qlbar[c_{\wh{G}^{\ad}}]\to \qlbar$ is obtained in this way.
\end{Emp}

\begin{Emp} \label{E:can}
{\bf Canonical homomorphism} (compare \cite[Section~5.3.1]{BV}). Fix $\gm\in G^{\on{sr}}(F)$.

\smallskip

(a) We have a natural embedding $\wh{G}_{\gm}/Z(\wh{G})\hra \wh{G}^{\ad}$, unique up to a conjugacy. Therefore composition $\wh{G}_{\gm}/Z(\wh{G})\hra \wh{G}^{\ad}\overset{\pr_{\wh{G}^{\ad}}}{\lra}c_{\wh{G}^{\ad}}$ and hence also composition
\[
\pr_{\gm}:(\wh{G}_{\gm})^{\Gm}/Z(\wh{G})^{\Gm}\hra \wh{G}_{\gm}/Z(\wh{G})\to c_{\wh{G}^{\ad}}
\]
is independent of all choices.

\smallskip

(b) The map $\pr_{\gm}$ of part~(a) induces an algebra homomorphism
\[
\on{can}_{\gm}:=\pr_{\gm}^*:\qlbar[c_{\wh{G}^{\ad}}]\to \qlbar[(\wh{G}_{\gm})^{\Gm}/Z(\wh{G})^{\Gm}].
\]
\end{Emp}

\begin{Emp} \label{E:ellcase}
{\bf The elliptic case.} Fix $\gm\in G(F)^{\on{ell}}$.

\smallskip

(a) In this case, the group $(\wh{G}_{\gm})^{\Gm}/Z(\wh{G})^{\Gm}$ is finite. Hence, by Section~\re{stconj}(c), the group 
$(\wh{G}_{\gm})^{\Gm}/Z(\wh{G})^{\Gm}$ is canonically identified with $A^{\vee}_{\gm}$, so the algebra of regular functions  $\qlbar[(\wh{G}_{\gm})^{\Gm}/Z(\wh{G})^{\Gm}]$ is canonically identified with the group algebra $\qlbar[A_{\gm}]$. Hence morphism
$\on{can}_{\gm}$ from Section~\re{can}(b) defines a homomorphism of $\qlbar$-algebras $\qlbar[c_{\wh{G}^{\ad}}]\to \qlbar[A_{\gm}]$.

\smallskip

(b) The action of the group $A_{\gm}$ on $\on{Fun}([\on{Orb}^{\st}_{\gm}])$ from Section~\re{notstconj}(a) induces an action of the group algebra $\qlbar[A_{\gm}]$, hence, by part~(a), an action of $\qlbar[c_{\wh{G}^{\ad}}]$.
\end{Emp}

\begin{Lem} \label{L:isot}
Let $\gm\in G(F)^{\on{ell}}$, and let $[s]$ be as in Section~\re{chev}(b). Then a function $f\in \on{Fun}([\on{Orb}^{\st}_{\gm}])$ belongs to $\on{Fun}([\on{Orb}^{\st}_{\gm}])_{[s]}$ (see Section~\re{notstconj}(c)) if and only if we have $h(f)=\zeta_{[s]}(h)\cdot f$ for every $h\in \qlbar[c_{\wh{G}^{\ad}}]$.
\end{Lem}

\begin{proof}
Every character $\xi\in A_{\gm}^{\vee}$ gives rise to a character $\xi$ of $\qlbar[A_{\gm}]$, hence restricts to a character
$\xi|_{\qlbar[c_{\wh{G}^{\ad}}]}$ of $\qlbar[c_{\wh{G}^{\ad}}]$. Next, using isomorphism 
$A^{\vee}_{\gm}\simeq (\wh{G}_{\gm})^{\Gm}/Z(\wh{G})^{\Gm}$ from Section~\re{ellcase}(a), a character $\xi\in A_{\gm}^{\vee}$ gives rise to a semisimple conjugacy class $[\xi]\subseteq\wh{G}^{\ad}$, hence to a character $\zeta_{[\xi]}$ (see Section~\re{chev}(b)) of $\qlbar[c_{\wh{G}^{\ad}}]$. Since $\xi|_{\qlbar[c_{\wh{G}^{\ad}}]}=\zeta_{[\xi]}$ for every $\xi\in A_{\gm}^{\vee}$, the assertion follows. 
\end{proof}

\begin{Emp}
{\bf Notation.} We have a restriction map $\qlbar[c_{\wh{G}^{\ad}}]\to \qlbar[\ov{S}_{\varphi}]^{\ov{S}_{\varphi}}$, corresponding to an inclusion $\ov{S}_{\varphi}\subseteq\wh{G}^{\ad}$.  In particular, the action of $\qlbar[\ov{S}_{\varphi}]^{\ov{S}_{\varphi}}$ on $\Theta_{\varphi}$ induces an action of $\qlbar[c_{\wh{G}^{\ad}}]$ on $\Theta_{\varphi}$.

\end{Emp}

As it will be explained in Section~\ref{HKW}, the following proposition is essentially a reformulation of results of \cite{HKW}.

\begin{Prop} \label{P:HKW}
The restriction map $\res_{\gm}:\Theta_{\varphi}\to \on{Fun}([\on{Orb}^{\st}_{\gm}])$ is $\qlbar[c_{\wh{G}^{\ad}}]$-equivariant.
\end{Prop}

\begin{Emp} \label{E:zemb}
{\bf $Z$-embeddings.}
(a) Let $\iota: G\to G_1$ be a $z$-embedding in the sense of \cite[Definition~5.1]{Ka}. In particular, $\iota$ is an embedding of
connected reductive groups over $F$ such that $Z(G_1)$ is connected, $\iota(Z(G))\subseteq Z(G_1)$ and $G_1=\iota(G)\cdot Z(G_1)$.

\smallskip

(b) A $z$-embedding $\iota$ induces a homomorphism $\wh{\iota}:{}^L G_1\to {}^L G_2$ of Langlands dual groups. Then, by \cite[Corollary~5.13]{Ka}, there exists a semisimple $L$-parameter $\varphi_1:W_F\to {}^L G_1$, lifting $\varphi$. Moreover,
homomorphism $\wh{\iota}$ restricts to a homomorphism $\wh{\iota}_{\varphi}:S_{\varphi_1}\to S_{\varphi}$ and induces an isomorphism $\wh{\iota}_{\varphi}:\ov{S}_{\varphi_1}\isom \ov{S}_{\varphi}$ (use \cite[Lemma~5.12]{Ka}). In particular, $L$-parameter $\varphi_1$ is elliptic, because $\varphi$ is such.

\smallskip

(c) By part~(b), a preimage $\wh{\iota}_{\varphi}^{-1}([\ka])$ of an $\ov{S}_{\varphi}$-conjugacy class $[\ka]\subseteq \ov{S}_{\varphi}$
is an $\ov{S}_{\varphi_1}$-conjugacy class $[\ka_1]\subseteq \ov{S}_{\varphi_1}$. Moreover, by a standard argument (see, for example, \cite[Lemma~4.1.14]{BV}), one shows that a function $f_1\in C^{\infty}_{\inv}(G_1(F)^{\on{ell}})$ is $\C{E}_{\varphi_1,[\ka_1]}$-stable if and only if its pullback  $f:=f_1|_{G(F)}\in C^{\infty}_{\inv}(G_1(F)^{\on{ell}})$ is $\C{E}_{\varphi,[\ka]}$-stable.

\smallskip

(d) The isomorphism $\wh{\iota}_{\varphi}$ from part~(b) induces an isomorphism of $\qlbar$-algebras $(\wh{\iota}_{\varphi})^*:\qlbar[\ov{S}_{\varphi}]^{\ov{S}_{\varphi}}\isom\qlbar[\ov{S}_{\varphi_1}]^{\ov{S}_{\varphi_1}}$.
Moreover, in the notation of part~(c), it satisfies $(\wh{\iota}_{\varphi})^*(1_{[\ka]})=1_{[\ka_1]}$.
\end{Emp}

As it will be explained in Section~\ref{zemb}, the following result is a formal consequence of the functoriality properties of \cite{FS} and results of \cite{Ka}.

\begin{Prop} \label{P:zemb}
Suppose that we are in the situation of Section~\re{zemb}.

\smallskip

(a) The restriction functor $\pi\mapsto \pi|_{G(F)}$ induces a bijection $\Pi^{FS}_{\varphi_1}\isom \Pi^{FS}_{\varphi}$. Hence, the restriction map $f\mapsto f|_{G(F)}$ induces an isomorphism $\iota^*:\Theta_{\varphi_1}\isom \Theta_{\varphi}$ of $\qlbar$-vector spaces.

\smallskip

(b) For every $f_1\in \Theta_{\varphi_1}$ and $h\in \qlbar[\ov{S}_{\varphi}]^{\ov{S}_{\varphi}}$, we have an equality
\[
h(\iota^*(f_1))=((\wh{\iota}_{\varphi})^*(h))(f_1).
\]

\end{Prop}

Now we are ready to prove our Main Theorem.

\begin{Emp}
\begin{proof}[Proof of \rt{main}] We divide our proof into three steps:

\smallskip

{\bf Step 1.} Assume first that $Z(G)$ is connected. In this case, we are going to deduce the $\C{E}_{\varphi,[\ka]}$-stability of $f_{[\ka]}^{\on{ell}}$  from \rp{HKW}.

By Section~\re{stable}(c), the function $f^{\on{ell}}_{[\ka]}\in  C^{\infty}_{\inv}(G(F)^{\on{ell}})$ is $\C{E}_{\varphi,[\ka]}$-stable if and only if
for every $\gm\in G(F)^{\on{ell}}$, its restriction $\res_{\gm}(f^{\on{ell}}_{[\ka]})=\res_{\gm}(f_{[\ka]})\in \on{Fun}([\on{Orb}^{\st}_{\gm}])$
belongs to $\on{Fun}([\on{Orb}^{\st}_{\gm}])_{[\ka]}$. Thus, by \rl{isot}, it suffices to show that $\res_{\gm}(f_{[\ka]})$ is a  $\qlbar[c_{\wh{G}^{\ad}}]$-eigenfunction with eigencharacter $\zeta_{[\ka]}$.

Then, by \rp{HKW}, it suffices to show that $f_{[\ka]}=1_{[\ka]}(f)$ is a $\qlbar[c_{\wh{G}^{\ad}}]$-eigenfunction with eigencharacter $\zeta_{[\ka]}$. Since for every $h\in \qlbar[c_{\wh{G}^{\ad}}]$ we have equalities
\[
h(1_{[\ka]}(f))=((h|_{\ov{S}_{\varphi}})\cdot 1_{[\ka]})(f)\text{ and }(h|_{\ov{S}_{\varphi}})\cdot 1_{[\ka]}=\zeta_{[\ka]}(h)\cdot 1_{[\ka]},
\]
the assertion follows.

\smallskip

{\bf Step 2.} For a general $G$, we deduce the $\C{E}_{\varphi,[\ka]}$-stability of $f_{[\ka]}^{\on{ell}}$ from the case of
groups with connected center using $z$-embeddings.

We choose a $z$-embedding $\iota: G\to G_1$ (use \cite[Corollary~5.3]{Ka}), and let $\varphi_1$ and $[\ka_1]\subseteq \ov{S}_{\varphi_1}$ be as in Sections~\re{zemb}(b),(c). By \rp{zemb}(a), there exists a unique element $f_1\in \Theta_{\varphi_1}$ such that $f_1|_{G(F)}=f$. Moreover, by \rp{zemb}(b), the identity $(\wh{\iota}_{\varphi})^*(1_{[\ka]})=1_{[\ka_1]}$ from Section~\re{zemb}(d) implies that $(f_1)_{[\ka_1]}|_{G(F)}=f_{[\ka]}$.

Since the center $Z(G_1)$ is connected, we conclude by Step 1 that each $(f_1)^{\on{ell}}_{[\ka_1]}$ is $\C{E}_{\varphi_1,[\ka_1]}$-stable. Hence, the function $f^{\on{ell}}_{[\ka]}=(f_1)^{\on{ell}}_{[\ka_1]}|_{G(F)}$ is $\C{E}_{\varphi,[\ka]}$-stable by Section~\re{zemb}(c).

\smallskip

{\bf Step 3.} Using the fact \cite[Theorem~IX.0.5(vii)]{FS} asserting that the correspondence $\pi\mapsto \varphi_{\pi}$ is compatible with parabolic induction, we see that all representations $\pi\in\Pi^{\FS}_{\varphi}$ are supercuspidal.
Since the endoscopic datum  $\C{E}_{\varphi,[\ka]}$ is elliptic (see Section~\re{enddatum}(d)) and the restriction $f^{\on{ell}}_{[\ka]}=(f_{[\ka]})|_{G(F)^{\on{ell}}}$ is $\C{E}_{\varphi,[\ka]}$-stable by Step 2, the $\C{E}_{\varphi,[\ka]}$-stability of $f_{[\ka]}$ now follows from \rp{ell}.
\end{proof}
\end{Emp}

\begin{Emp}
{\bf Remark.} The structure of the proof of the $\C{E}_{\varphi,[\ka]}$-stability of $f^{\on{ell}}_{[\ka]}$ is analogous to the one of \cite[Theorem~4.4.9(b)]{BV}, where another instance of stability is shown by a geometric method. Namely, in both cases the assertion is reduced to the case of groups with connected center, while in this case the assertion is deduced from an assertion that two actions are {\em compatible}.

\smallskip

More precisely, we deduce $\C{E}_{\varphi,[\ka]}$-stability of $f^{\on{ell}}_{[\ka]}$ from \rp{HKW}, while the result of \cite{BV} is deduced from \cite[Theorem~2.3.4]{BV}, and both \rp{HKW} and \cite[Theorem~2.3.4]{BV} can be stated as an assertion that natural actions of $\qlbar[c_{\wh{G}}]$ are $\qlbar[\pi_0(LG_{\gm})]$ are compatible via homomorphism $\on{can}_{\gm}$.
\end{Emp}

\section{Spectral action} In this sections all categories will be $\infty$-categories and all functors will be functors of $\infty$-categories.

\begin{Emp} \label{E:FS}
{\bf Spectral action of Fargues--Scholze (\cite{FS}).}

\smallskip

(a) Let $\Bun_G$ be the stack of $G$-bundles over the Fargues--Fontaine
curve. Let $D(\Bun_G):=D_{\lis}(\Bun_G,\qlbar)$ be the ($\infty$-)derived category of $\ell$-adic sheaves on $\Bun_G$, and let $D(\Bun_G)^{\om}$
be its full subcategory of compact objects.

\smallskip

(b) Let $\Par_{G}:=[Z^1(W_F ,\wh{G})/\wh{G}]$ be the stack of Langlands parameters (see \cite[Theorem~VII.1.3]{FS}), and let
$\Perf(\Par_{G})$ be the symmetric monoidal category of perfect complexes on $\Par_G$.
Fargues and Scholze defined an action of $\Perf(\Par_G)$
on $D(\Bun_G)^{\om}$ (see \cite[Corollary~X.1.3]{FS}), usually called the {\em spectral action}.

\smallskip

(c) Recall (see \cite[Proposition~VII.7.4]{FS}) that the category  $D(\Bun_G)$ is compactly generated. Therefore passing to ind-completions, the spectral action of part~(b) induces the action of $\QCoh(\Par_G)=\Ind\Perf(\Par_G)$
on $D(\Bun_G)$.

\smallskip

(d) The spectral action of part~(c) induces a homomorphism
\[
\Psi:\qlbar[\Par_G]=\End(1_{\QCoh(\Par_G)})\to \C{Z}(D(\Bun_G)):=\pi_0(\End(\Id_{D(\Bun_G)}))
\]
of $\qlbar$-algebras from the algebra of regular functions on
$\Par_G$ to the Bernstein center of $D(\Bun_G)$.
\end{Emp}

\begin{Emp} \label{E:connected}
{\bf The connected component of $\varphi$.}

\smallskip

(a) Since $L$-parameter $\varphi$ is elliptic, the connected component $C_{\varphi}\subseteq \Par_G$  of $\varphi$ consists of unramified twists of $\varphi$ (see \cite[Section~X.2]{FS} and \cite[Footnote~6 on page~12]{HJ}), that is, consists of $L$-parameters $\varphi\cdot\mu$, where $\mu:W_F\to \wh{Z}:=(Z(\wh{G})^{\Gm})^0$ is an unramified character. In particular, every $L$-parameter $\varphi'\in C_{\varphi}$ is elliptic.

\smallskip

(b) Note that $C_{\varphi}$ gives rise a direct summand $D^{C_{\varphi}}(\Bun_G)\subseteq D(\Bun_G)$. Explicitly,
let  $1_{C_{\varphi}}\in \qlbar[\Par_G]$ be the characteristic function of $C_{\varphi}$. Then
$D^{C_{\varphi}}(\Bun_G)$ consists of objects $A\in D(\Bun_G)$ such that $\Psi(1_{C_{\varphi}})\in \C{Z}(D(\Bun_G))$ acts on $A$ as identity.

\smallskip

(c) Using the fact that $\QCoh(\Par_G)$ is a {\em symmetric} monoidal category,
the spectral action of Section~\re{FS} restricts to an action of  $\QCoh(\Par_G)$ on $D^{C_{\varphi}}(\Bun_G)$ and induces an action of
$\QCoh(C_{\varphi})$ on $D^{C_{\varphi}}(\Bun_G)$. Furthermore, the latter action restricts to an action of
 $\Perf(C_{\varphi})$ on $D^{C_{\varphi}}(\Bun_G)^{\om}$.
%
%
%
\end{Emp}

\begin{Emp} \label{E:Hecke}
{\bf Hecke action.}

\smallskip

(a) Consider projection $\pr:\Par_G=[Z^1(W_F ,\wh{G})/\wh{G}]\to [\pt/\wh{G}]\to [\pt/\wh{G}^{\ad}]$, where $\pt:=\Spec\qlbar$. Then the pullback
$\QCoh([\pt/\wh{G}^{\ad}])\to\QCoh(\Par_G)$ is symmetric monoidal and maps $\Perf([\pt/\wh{G}^{\ad}])\to\Perf(\Par_G)$.
So the spectral action of Section~\re{FS}(c) induces an action  $\QCoh([\pt/\wh{G}^{\ad}])$ on $D(\Bun_G)$, called the {\em Hecke action}. Furthermore, the Hecke action  preserves the full
subcategory $D^{C_{\varphi}}(\Bun_G)$ and restricts to an action of
$\Perf([\pt/\wh{G}^{\ad}])$ on $D(\Bun_G)^{\om}$ and $D^{C_{\varphi}}(\Bun_G)^{\om}$.

\smallskip

(b) Let $\Rep_{\on{fd}}(\wh{G}^{\ad})$ be the category of finite-dimensional representations of $\wh{G}^{\ad}$. Then the action of part~(a) defines the action of the symmetric monoidal category $D^b(\Rep_{\on{fd}}(\wh{G}^{\ad}))=\Perf([\pt/\wh{G}^{\ad}])$, also called the {\em Hecke action}.
\end{Emp}

\begin{Emp} \label{E:strata}
{\bf Stratification.}

\smallskip

(a) Recall that $\Bun_G$ has a natural stratification $\{\Bun_G^b\}_b$, indexed by isocrystals $b\in B(G)$, and that each embedding $i_b:\Bun_G^b\hra \Bun_G$ is locally closed. Moreover, $i_b$ is an open embedding, if $b\in B(G)$ is basic.

\smallskip

(b) By part~(a), the union $\Bun^{\on{ss}}_G:=\bigcup_{b\in B(G)_{\on{basic}}}\Bun_G^b\subseteq\Bun_G$ is open and decomposes as a disjoint union $\Bun^{\on{ss}}_G:=\bigsqcup_{b\in B(G)_{\on{basic}}}\Bun_G^b$.

\smallskip

(c) Every $b\in B(G)$ gives rise to a reductive group $G_b$ over $F$, and $G_b=G$ if $b=1$.
Moreover, the category $D(\Bun_G^b):=D_{\on{lis}}(\Bun_G^b,\qlbar)$ is canonically identified with the derived category $D(G_b(F))$ of smooth representations of $G_b(F)$ over $\qlbar$ (see \cite[Propostion~VII.7.1]{FS}).

\smallskip

(d) Each pullback $i_b^*: D(\Bun_G)\to D(\Bun_G^b)\simeq D(G_b(F))$ has a fully faithful left adjoint
$i_{b\natural}:D(G_b(F))\hra D(\Bun_G)$ (see \cite[Propostion~VII.7.2]{FS}) and a fully faithful right adjoint
$i_{b*}:D(G^b(F))\hra D(\Bun_G)$. By adjunction, we have a canonical morphism of functors $i_{b\natural}\to i_{b*}$.

\smallskip

(e) The fully faithful functor $i_{b\natural}$ induces a restriction map between Bernstein centers $\C{Z}(D(\Bun_G))\to \C{Z}(D(G_b(F)))$. In particular, the algebra homomorphism $\Psi:\qlbar[\Par_G]\to \C{Z}(D(\Bun_G))$ from Section~\re{FS}(d) induces an algebra homomorphism
$\Psi^b:\qlbar[\Par_G]\to \C{Z}(D(G_b(F)))$ for all $b\in B(G)$.

\smallskip

(f) As in Section~\re{connected}(b), algebra homomorphism from part~(e) gives rise to a direct summand $D^{C_{\varphi}}(G_b(F))\subseteq D(G_b(F))$ for every $b\in B(G)$. 

\smallskip

(g) By construction, the functor $i_{b\natural}$ maps $D^{C_{\varphi}}(G_b(F))$ to  $D^{C_{\varphi}}(\Bun_G)$. Moreover,
using counit map $i_{b\natural}i_b^*\to\Id$ and fully faithfulness of $i_{b\natural}$, we conclude that $i_{b}^*$ maps $D^{C_{\varphi}}(\Bun_G)$ to  $D^{C_{\varphi}}(G_b(F))$. In particular, the induced functors $i_{b\natural}:D^{C_{\varphi}}(G_b(F))\to D^{C_{\varphi}}(\Bun_G)$ and
$i_{b}^*:D^{C_{\varphi}}(\Bun_G)\to D^{C_{\varphi}}(G_b(F))$ form an adjoint pair. Furthermore, using morphism $i_{b\natural}\to i_{b*}$
and fully-faithfulness of $i_{b*}$ we conclude that $i_{b*}$ maps $D^{C_{\varphi}}(G_b(F))$ to $D^{C_{\varphi}}(\Bun_G)$.

\smallskip

(h) Since every $\varphi'\in C_{\varphi}$ is elliptic (see Section~\re{connected}(a)), we deduce from \cite[Theorem~IX.7.2]{FS}
that $D^{C_{\varphi}}(G_b(F))=0$ if $b\in B(G)$ is not basic.

\end{Emp}








\begin{Lem} \label{L:ell}
Let $b\in B(G)$ be basic.

\smallskip

(a) For every $A\in D^{C_{\varphi}}(G_b(F))$, the canonical morphism $i_{b\natural}(A)\to i_{b*}(A)$ (see Section~\re{strata}(d)) is an isomorphism.

\smallskip

(b) The Hecke action on $\QCoh([\pt/\wh{G}^{\ad}])$ on $D^{C_{\varphi}}(\Bun_G)$ preserves the full subcategory
$i_{b\natural}(D^{C_{\varphi}}(G_b(F)))\subseteq D^{C_{\varphi}}(\Bun_G)$.
\end{Lem}
\begin{proof}
(a) We have to show that $i^*_{b'}(i_{b*}(A))\simeq 0$ for every $b'\neq b\in B(G)$. Since $A\in D^{C_{\varphi}}(G_b(F))$, we conclude from Section~\re{strata}(g) that
$i^*_{b'}(i_{b*}(A))\in D^{C_{\varphi}}(G_{b'}(F))$. Therefore it follows from Section~\re{strata}(h) that $i^*_{b'}(i_{b*}(A))\simeq 0$, if $b'$ is not basic. Finally, the assertion  $i^*_{b'}(i_{b*}(A))\simeq 0$ for basic $b'\neq b$ follows from the fact that the embedding  $i^{\on{ss}}_b:\Bun_G^b\hra\Bun_G^{\on{ss}}$ is open and closed (see Section~\re{strata}(b)).

\smallskip

(b) Note that there is a natural bijection $\kappa:\pi_0(\Bun_G)\isom X^*(Z(\wh{G}^{\Gm}))$ (see \cite[Corollary~IV.1.23]{FS}). Moreover, by  \cite[Lemma~5.3.2]{Zou}
for each connected component $\Bun_G^{\al}\subseteq \Bun_G$  the action of $\QCoh([\pt/\wh{G}^{\ad}])$ on $D^{C_{\varphi}}(\Bun_G)$ preserves full subcategory $D(\Bun^{\al}_G)\subseteq  D(\Bun_G)$. Hence it preserves a full subcategory
\[
D^{C_{\varphi}}(\Bun^{\al}_G):=D^{C_{\varphi}}(\Bun_G)\cap D(\Bun^{\al}_G)\subseteq  D^{C_{\varphi}}(\Bun_G).
\]

Let  $\Bun_G^{\al}$ be a connected component of $\Bun_G$ such that $\Bun^b_G\subseteq\Bun_G^{\al}$.
It suffices to show that the functor $i_{b\natural}: D(G_b(F))\to D(\Bun^{\al}_G)$ induces an equivalence of categories
$D^{C_{\varphi}}(G_b(F))\isom D^{C_{\varphi}}(\Bun^{\al}_G)$.

We have to show that for every $A\in D^{C_{\varphi}}(\Bun^{\al}_G)$ and $b'\neq b\in B(G)$ such that $\Bun^{b'}_G\subseteq\Bun_G^{\al}$ we have $i_{b'}^*(A)\simeq 0$. But such $b'$ is automatically non-basic, so the assertion follows from the fact
that $i_{b'}^*(A)\in D^{C_{\varphi}}(\Bun^{b'}_G)$ (by Section~\re{strata}(g)) and hence $i_{b'}^*(A)\simeq 0$ (by Section~\re{strata}(h)).
\end{proof}




\begin{Emp} \label{E:spectral}
{\bf The spectral action of $\QCoh([\pt/S_{\varphi}])$.}

\smallskip

(a) Let $[C_{\varphi}]$ be the coarse moduli space of $C_{\varphi}$. Explicitly, let $\wt{C}_{\varphi}\subseteq Z^1(W_F ,\wh{G})$ be the preimage of $C_{\varphi}\subseteq\Par_G$.
Then $\wt{C}_{\varphi}$ is an affine scheme (see \cite[Theorem~VIII.1.3]{FS}), $C_{\varphi}$ is the quotient stack $[\wt{C}_{\varphi}/\wh{G}]$ and $[C_{\varphi}]$ is an affine scheme, obtained as a GIT quotient
$\wt{C}_{\varphi}//\wh{G}$.

\smallskip

(b) Consider the (homotopy) fiber product $C^{\varphi}:=C_{\varphi}\times_{[C^{\varphi}]}\pt$ of the canonical projection $\pr:C_{\varphi}\to [C^{\varphi}]$ and a morphism $\pt\to [C^{\varphi}]$ corresponding to an element $[\varphi]:=\pr(\varphi)\in [C^{\varphi}](\qlbar)$.

\smallskip

(c) Since $L$-parameter $\varphi$ is assumed to be elliptic, we conclude that $C^{\varphi}\simeq [\pt/S_{\varphi}]$ (see \cite[Lemmas 3.6 and 3.10]{HJ}).

\smallskip

(d) Morphisms $\pr$ and $[\varphi]$ from part~(b) induces symmeric monoidal functors $\pr^*: \QCoh([C_{\varphi}])\to \QCoh(C_{\varphi})$ and
$[\varphi]^*: \QCoh([C_{\varphi}])\to \QCoh(\pt)$. In particular, the action of $\QCoh(C_{\varphi})$ on $D^{C_{\varphi}}(\Bun_G)$ from Section~\re{connected}(c) restricts to the action of $\QCoh([C_{\varphi}])$. Taking Lurie tensor product $-\otimes_{\QCoh([C_{\varphi}])}\QCoh(\pt)$, we get an action of the symmetric monoidal category
$\QCoh(C_{\varphi})\otimes_{\QCoh([C_{\varphi}])}\QCoh(\pt)$ on
\[
D^{\varphi}(\Bun_G):=D^{C_{\varphi}}(\Bun_G)\otimes_{\QCoh([C_{\varphi}])}\QCoh(\pt).
\]

\smallskip

(e) Using the fact that the morphism $\pt\to [C_{\varphi}]$ is affine, we conclude from the Barr-Beck-Lurie theorem \cite[Theorem~4.7.3.5]{Lu} that the natural morphism
\[
\QCoh(C_{\varphi})\otimes_{\QCoh([C_{\varphi}])}\QCoh(\pt)\to \QCoh(C^{\varphi})
\]
is an equivalence.

\smallskip

(f) Combining identifications of parts~(c) and (e), the action of part~(d) induces an action of $\QCoh([\pt/S_{\varphi}])$ on $D^{\varphi}(\Bun_G)$.
\end{Emp}

\begin{Emp} \label{E:compat}
{\bf Compatibilities.}

\smallskip

(a) The functor $[\varphi]^*: \QCoh([C_{\varphi}])\to \QCoh(\pt)$ is (symmetric) monoidal, so it induces an action $\QCoh([C_\varphi])$ on $\QCoh(\pt)$. Moreover, by the projection formula, the functor  $[\varphi]_*: \QCoh(\pt)\to \QCoh([C_\varphi])$, right adjoint to $[\varphi]^*$, is
 $\QCoh([C_{\varphi}])$-linear. Therefore it induces a functor $[\varphi]_*:D^{\varphi}(\Bun_G)\to D^{C_\varphi}(\Bun_G)$, right adjoint to the pullback $[\varphi]^*:D^{C_\varphi}(\Bun_G)\to D^{\varphi}(\Bun_G)$.

\smallskip

(b) The inclusion ${S}_{\varphi}\hra\wh{G}$ induces a monoidal functor 
\[
\QCoh([\pt/\wh{G}])\to \QCoh([\pt/S_{\varphi}]).
\] 
Therefore the spectral action  of $\QCoh([\pt/S_{\varphi}])$ on $D^{\varphi}(\Bun_G)$
induces an action of  $\QCoh([\pt/\wh{G}])$ on  $D^{\varphi}(\Bun_G)$. Moreover, using commutative diagram of projections
\[
\begin{CD}
C^{\varphi} @>>> C_{\varphi}\\
@VVV @VVV\\
[\pt/{S}_{\varphi}] @>>> [\pt/\wh{G}], 
\end{CD}
\]
we conclude that the pullback $[\varphi]^*:D^{C_\varphi}(\Bun_G)\to D^{\varphi}(\Bun_G)$ is $\QCoh([\pt/\wh{G}])$-linear.

\smallskip

(c) Moreover, since the monoidal $\infty$-category $\QCoh([\pt/\wh{G}]$ is rigid, we conclude from parts~(a) and (b) that the push-forward  $[\varphi]_*:D^{\varphi}(\Bun_G)\to D^{C_\varphi}(\Bun_G)$
is  $\QCoh([\pt/\wh{G}])$-linear as well (see \cite[I, Lemma~9.3.6]{GR}).
\end{Emp}

\begin{Emp} \label{E:obssp}
{\bf Observations.}

\smallskip

(a) Since $[C_{\varphi}]$ is affine and the functor $i_{1\natural}: D^{C_{\varphi}}(G(F))\to  D^{C_{\varphi}}(\Bun_G)$
is fully faithful, we conclude as in \cite[Remark before Proposition~2.7]{HJ} that there exists a unique action of $\QCoh([C_{\varphi}])$ on $D^{C_{\varphi}}(G(F))$ such that functor $i_{1\natural}$ is $\QCoh([C_{\varphi}])$-linear. Moreover, since monoidal category $\QCoh([C_{\varphi}])$ is rigid, the right adjoint $i_{1}^*$ of $i_{1\natural}$
is $\QCoh([C_{\varphi}])$-linear as well (see \cite[I, Lemma~9.3.6]{GR}).

\smallskip

(b) Using part~(a) and applying functor $-\otimes_{\QCoh([C_{\varphi}])}\QCoh(\pt)$, we can form a category $D^{\varphi}(G(F)):=D^{C_{\varphi}}(G(F))\otimes_{\QCoh([C_{\varphi}])}\QCoh(\pt)$ and a pair of adjoint functors  $i_{1\natural}: D^{\varphi}(G(F))\to  D^{\varphi}(\Bun_G)$ and $i_{1}^*: D^{\varphi}(\Bun_G)\to  D^{\varphi}(G(F))$, and functor $i_{1\natural}$ is fully faithful. In particular, functor  $i_{1\natural}$ induces a functor $D^{\varphi}(G(F))^{\om}\to  D^{\varphi}(\Bun_G)^{\om}$ between subcategories of compact objects.

\smallskip


\smallskip

(c) The spectral action of $\QCoh([\pt/S_{\varphi}])$ on $D^{\varphi}(\Bun_G)$ from Section~\re{spectral} restricts to an action of its monoidal subcategory $\Perf([\pt/S_{\varphi}])$. Furthermore, we have a symmetric monoidal functor  $\QCoh([\pt/\ov{S}_{\varphi}])\to\QCoh([\pt/S_{\varphi}])$, corresponding to the projection $S_{\varphi}\to \ov{S}_{\varphi}$, thus an action of $\QCoh([\pt/\ov{S}_{\varphi}])$ on $D^{\varphi}(\Bun_G)$.

\smallskip

\end{Emp}

\begin{Lem} \label{L:compact}
(a) The spectral action of $\Perf([\pt/S_{\varphi}])$  on  $D^{\varphi}(\Bun_G)$ from Section~\re{spectral} preserves the subcategory $D^{\varphi}(\Bun_G)^{\omega}\subseteq D^{\varphi}(\Bun_G)$ of compact objects.

\smallskip

(b) Moreover, the induced action of $\Perf([\pt/\ov{S}_{\varphi}])$ on $D^{\varphi}(\Bun_G)^{\omega}$ preserves the full subcategory
$i_{1\natural}(D^{\varphi}(G(F))^{\omega})$.

\end{Lem}

\begin{proof}
(a) Note that the pullback $[\varphi]^*: \QCoh([C_{\varphi}])\to \QCoh(\pt)$ has a conservative right adjoint $[\varphi]_*$.
Therefore  $\QCoh(\pt)$ is generated by the image of $[\varphi]^*$ under colimits. Hence for every $\QCoh([C_{\varphi}])$-module $\C{M}$ the tensor product
\[
\C{M}\otimes_{\QCoh([C_{\varphi}])}\QCoh(\pt)
\]
is generated by the image of the functor
$\C{M}\to\C{M}\otimes_{\QCoh([C_{\varphi}])}\QCoh(\pt)$ under colimits.

\smallskip 

Applying the above observation to the $\QCoh([C_{\varphi}])$-module $\QCoh(C_{\varphi})$, we conclude that $\QCoh([\pt/S_{\varphi}])\simeq \QCoh(C^{\varphi})$ is generated by the image of the pullback $\QCoh(C_{\varphi})\to \QCoh(C^{\varphi})$ by colimits. Hence $\Perf(C^{\varphi})$ is generated by the image of
pullback $\Perf(C_{\varphi})\to \Perf(C^{\varphi})$ under finite colimits and retracts. Thus it suffices to show that  $D^{\varphi}(\Bun_G)^{\omega}\subseteq D^{\varphi}(\Bun_G)$ is stable under the action
of the image of  $\Perf(C_{\varphi})\to \Perf(C^{\varphi})$.

\smallskip 

Next, arguing as above, $D^{\varphi}(\Bun_G)^{\omega}$ is generated by the essential image of  $D^{C_{\varphi}}(\Bun_G)^{\omega}\to D^{\varphi}(\Bun_G)^{\omega}$ under finite colimits and retracts. So it remains to show that the image of  $D^{C_{\varphi}}(\Bun_G)^{\omega}\to D^{\varphi}(\Bun_G)^{\omega}$ is stable under the image of  $\Perf(C_{\varphi})\to \Perf(C^{\varphi})$. Since   $D^{C_{\varphi}}(\Bun_G)^{\omega}\subseteq D^{C_{\varphi}}(\Bun_G)$ is stable under the action of the image of  $\Perf(C_{\varphi})$ (see Section~\re{connected}(c)), the assertion follows.

\smallskip

(b) Using Section~\re{compat}(b) and arguing as in part~(a), it suffices to show that the action of $\Perf([\pt/\wh{G}^{\ad}])$ on $D^{C_{\varphi}}(\Bun_G)^{\omega}$ preserves a full subcategory $i_{1\natural}(D^{C_{\varphi}}(G(F))^{\omega})$.
Thus, by \rl{ell}(b), it suffices to show the equality
\[
i_{1\natural}(D^{C_{\varphi}}(G(F))^{\omega})=i_{1\natural}(D^{C_{\varphi}}(G(F)))\cap  D^{C_{\varphi}}(\Bun_G)^{\omega}.
\]
The inclusion $\subseteq$ was mentioned before, while inclusion $\supseteq$ follows from the fact that functor 
$i_{1\natural}:D^{C_{\varphi}}(G(F))\to D^{C_{\varphi}}(\Bun_G)$ is fully faithful and its right adjoint $i_1^*:D^{C_{\varphi}}(\Bun_G)\to D^{C_{\varphi}}(G(F))$ preserve compact objects (compare Section~\re{obssp}(b)).
\end{proof}

%
%

\begin{Emp} \label{E:main}
{\bf Main construction.}

\smallskip

(a) By \rl{compact}, the action of $\Perf([\pt/\ov{S}_{\varphi}])$ on  $D^{\varphi}(\Bun_G)$
induces an action of $\Perf([\pt/\ov{S}_{\varphi}])$ on   $D^{\varphi}(G(F))^{\om}$ .


\smallskip

(b)  The action of part~(a) induces an action of the $\qlbar$-algebra $K_0(\Perf([\pt/\ov{S}_{\varphi}]))$ on $K_0(D^{\varphi}(G(F))^{\om})$, where for every triangulated or stable $\infty$-category $\C{C}$, we denote by $K_0(\C{C})$ its Grothendieck group, tensored over $\qlbar$.

\smallskip

(c) The  assignment $V\mapsto\Tr(-,V)$ identifies the $\qlbar$-algebra 
\[
K_0(D^b(\Rep_{\on{fd}}(\ov{S}_{\varphi})))\simeq K_0(\Perf([\pt/\ov{S}_{\varphi}]))
\] 
with $\qlbar[\ov{S}_{\varphi}]^{\ov{S}_{\varphi}}$. On the other hand, by \rl{kgroups} below, we have a natural isomorphism of $\qlbar$-vector spaces $K_0(D^{\varphi}(G(F))^{\om})\simeq \Theta_{\varphi}$.

\smallskip

(d) Using identifications of part~(c), the action of part~(b) induces an action of the $\qlbar$-algebra $\qlbar[\ov{S}_{\varphi}]^{\ov{S}_{\varphi}}$ on $\Theta_{\varphi}$.
\end{Emp}

\begin{Lem} \label{L:kgroups}
We have a natural isomorphism $K_0(D^{\varphi}(G(F))^{\om})\simeq \Theta_{\varphi}$ of vector spaces.
\end{Lem}

\begin{proof}
Arguing as in \cite[Corollary~2.9]{HJ}, every $\pi\in\Pi^{FS}_{\varphi}\in D^{C_\varphi}(G(F))$ has a natural lift to an object $\wt{\pi}\in D^{\varphi}(G(F))$.
Moreover, by \cite[Lemma~3.10]{HJ}, every object of $D^{\varphi}(G(F))^{\om}$ is isomorphic to a direct sum $\bigoplus_{\al=1}^n\wt{\pi}_\al[n_\al]$ for some $\pi_\al\in\Pi^{FS}_{\varphi}$ and $n_{\al}\in\B{Z}$. Therefore $K_0(D^{\varphi}(G(F))^{\om})$ is naturally identified with a $\qlbar$-vector space $\qlbar[\Pi^{FS}_{\varphi}]$ with basis $\Pi^{FS}_{\varphi}$. Hence, using linear independence of characters, the map $\pi\mapsto \chi_{\pi}$ induces an isomorphism of $\qlbar$-vector spaces $K_0(D^{\varphi}(G(F))^{\om})\simeq \Theta_{\varphi}$.
\end{proof}

\begin{Emp}
{\bf Remark.} Note that the situation drastically simplifies when the center of $G(F)$ is compact. 
Indeed, in this case, the group invariants $Z(\wh{G})^{\Gm}$ is finite, the map 
$[\varphi]:\pt\to [C_{\varphi}]$ from Section~\re{spectral}(b) is an isomorphism 
(see Section~\re{connected}(a)), morphism $C^{\varphi}\to C_{\varphi}$ is an isomorphism, and functors 
$D^{C_{\varphi}}(\Bun_G)\to D^{\varphi}(\Bun_G)$ and  $D^{C_{\varphi}}(G(F))\to D^{\varphi}(G(F))$ are equivalences of categories, 
which makes Sections \re{spectral}-\re{obssp} and \rl{compact} vacuous.
\end{Emp}

\section{Proof of \rp{HKW}} \label{HKW}

\begin{Emp}
{\bf Reformulation of the problem.}

(a) For every irreducible representation $V\in\on{Irr}(\wh{G}^{\ad})$, we set
\[
[V]:=\Tr(-, V)\in \qlbar[\wh{G}^{\ad}]^{\wh{G}^{\ad}}=\qlbar[c_{\wh{G}^{\ad}}].
\]
Since such $[V]$'s generate $\qlbar[c_{\wh{G}^{\ad}}]$ as a $\qlbar$-vector space, it suffices to show that for every $V\in\on{Irr}(\wh{G})$,
$\pi\in \Pi^{\FS}_{\varphi}$ and $\gm\in G(F)^{\on{ell}}$, we have an equality
\begin{equation} \label{Eq:HKW}
\left([V](\chi_{\pi})\right)(\gm)=\left(\on{can}_{\gm}([V])(\chi_{\pi}|_{\on{Orb}^{\st}_{\gm}})\right)(\gm).
\end{equation}

\smallskip

(b) Our identity \form{HKW} is a formal consequence of results of \cite{HKW}. Namely, we can assume that $V=V_{\mu}^{\vee}$, where
$V_{\mu}$ is an irreducible representation of $\wh{G}$ of highest weight $\mu$. Then, by \cite[Theorem~6.5.2]{HKW}, it suffices to show  equalities
\begin{equation} \label{Eq:HKW1}
\left([V](\chi_{\pi})\right)(\gm)=\chi_{\on{Mant}_{1,\mu}(\pi)}(\gm),
\end{equation}
and
\begin{equation} \label{Eq:HKW2}
\left(\on{can}_{\gm}([V])(\chi_{\pi}|_{\on{Orb}^{\st}_{\gm}})\right)(\gm)=\left(\C{T}^{G\to G}_{1,\mu}(\chi_{\pi})\right)(\gm),
\end{equation}
where $\on{Mant}_{1,\mu}(\pi)$ is defined in \cite[paragraph after Definition~2.4.3]{HKW},
while $\C{T}^{G\to G}_{1,\mu}$ is defined in \cite[Definition~3.2.7]{HKW}.
\end{Emp}

\begin{Emp}
{\bf Proof of identity \form{HKW1}.}
By \rl{ell}, there exists a unique action of $\Perf([\pt/\wh{G}^{\ad}])$ on $D(G(F))^{\om}$ such that
$i_{1*}: D^{C_\varphi}(G(F))^{\om}\to  D^{C_\varphi}(\Bun_G)^{\om}$ is $\Perf([\pt/\wh{G}^{\ad}])$-linear. Moreover, since Hecke operators preserve ULA objects (see \cite[Theorem~IX.2.2]{FS} and compare \cite[Definition~1.6.2]{Ha}), we conclude that $V(\pi)\in D(G(F))^{\om}$ has the property that $H^*(V(\pi)):=\bigoplus_n H^n(V(\pi))$ is a smooth
representation of $G(F)$ of finite length. Therefore we can form an invariant generalized function  $\chi_{V(\pi)}\in\Theta_{\varphi}$.
Furthermore, using observation of Section~\re{compat}(c) and unwinding the construction of Section~\re{main} we have an equality $[V](\chi_{\pi})=\chi_{V(\pi)}$.
Now identity \form{HKW1} follows from \cite[Proposition~6.4.5]{HKW}.

\end{Emp}

\begin{Emp}
{\bf Proof of identity \form{HKW2}.}
\smallskip

(a) Let $\La:=X^*(\wh{G}_{\gm}/Z(\wh{G}))$ be the group of characters. For every representation $U$ of $\wh{G}^{\ad}$ and every $\la\in\La$, we denote by $r_{U}[\la]$ the dimension of the $\la$-weight subspace $U|_{\wh{G}_{\gm}/Z(\wh{G})}$. Notice that the embedding
$\wh{G}_{\gm}/Z(\wh{G})\hra \wh{G}^{\ad}$ is defined uniquely up to a conjugacy, and therefore the dimension $r_{U}[\la]$ is well-defined.

\smallskip

(b) By definition, $\on{can}_{\gm}([V])\in \qlbar[(\wh{G}_{\gm})^{\Gm}/Z(\wh{G})^{\Gm}]$ is equals to the restriction of
$\Tr(-,V)|_{\wh{G}_{\gm}/Z(\wh{G})}=\sum_{\la\in\La}r_{V}[\la]\la\in \qlbar[\wh{G}_{\gm}/Z(\wh{G})]$ to $(\wh{G}_{\gm})^{\Gm}/Z(\wh{G})^{\Gm}$.

\smallskip

(c) Notice that the isomorphism $A_{\gm}^{\vee}\simeq (\wh{G}_{\gm})^{\Gm}/Z(\wh{G})^{\Gm}$ (see Section~\re{stconj}(c)) induces an isomorphism
$A_{\gm}\simeq X^*((\wh{G}_{\gm})^{\Gm}/Z(\wh{G})^{\Gm})$. In particular, the embedding $(\wh{G}_{\gm})^{\Gm}/Z(\wh{G})^{\Gm}\hra \wh{G}_{\gm}/Z(\wh{G})$ induces a (surjective) homomorphism $\La\to A_{\gm}$ on characters. Therefore the action of $A_{\gm}$ on
$\on{Fun}([\on{Orb}^{\st}_{\gm}])$ (see Section~\re{notstconj}(a)) induces an action of $\La$ on $\on{Fun}([\on{Orb}^{\st}_{\gm}])$.

\smallskip

(d) Unwinding the definitions, we have an equality
\[
\left[\C{T}^{G\to G}_{1,\mu}(\chi_{\pi})\right](\gm)=\left[\left(\sum_{\la\in\La}r_{V_{\mu}}[\la] \la^{-1}\right)(\chi_{\pi}|_{\on{Orb}^{\st}_{\gm}})\right](\gm)=\left[\left(\sum_{\la\in\La}r_{V_{\mu}}[\la^{-1}]\la\right)(\chi_{\pi}|_{\on{Orb}^{\st}_{\gm}})\right](\gm).
\]
Notice that sign $(-1)^d$ in the definition of $\C{T}^{G\to G}_{1,\mu}$ is always $1$ in our case, because
$\mu$ is a dominant weight of $\wh{G}^{\ad}$.

\smallskip

(e) Since $r_{V_{\mu}}[\la^{-1}]=r_{V_{\mu}^{\vee}}[\la]=r_{V}[\la]$, identity \form{HKW2} follows from  parts (b),(d).
\end{Emp}

\section{Proof of \rp{zemb}} \label{zemb}

\begin{Emp}
{\bf Proof of part~(a).} We are going to deduce the result from a combination of
\cite[Theorem~IX.0.5(iii),(v)]{FS} and \cite[Fact~5.5 and Corollary~5.13]{Ka}.

\smallskip

(i) By \cite[Fact~5.5]{Ka}, we have $G_1(F) = Z(G_1)(F)\cdot G(F)$.

\smallskip

(ii) For every $\pi\in\Pi^{FS}_{\varphi_1}$, the restriction $\pi|_{G(F)}$ is irreducible by part~(i). Thus, by \cite[Theorem~IX.0.5(v)]{FS}, we have $\pi|_{G(F)}\in \Pi^{FS}_{\varphi}$.

\smallskip

(iii) By \cite[Theorem~IX.0.5(iii)]{FS}, for every $\pi',\pi''\in\Pi^{FS}_{\varphi_1}$ we have an equality $\pi'|_{Z(G_1(F))}=\pi''|_{Z(G_1(F))}$. Hence, by part~(i), we conclude that $\pi'\simeq \pi''$ if and only if $\pi'|_{G(F)}\simeq\pi''|_{G(F)}$.

\smallskip

(iv) It remains to show that for every  $\pi\in\Pi^{FS}_{\varphi}$ there exists  $\pi_1\in\Pi^{FS}_{\varphi_1}$ such that
$\pi_1|_{G(F)}=\pi$. Let $z_{\varphi_1}:Z(G_1(F))\to\qlbar^{\times}$ be the character, corresponding to $\varphi_1$. Then, by
\cite[Theorem~IX.0.5(iii)]{FS}, we have $\pi|_{Z(G(F))}= z_{\varphi_1}|_{Z(G(F))}\cdot\Id$. Hence, by part~(i), there exists
a unique $\pi'\in \Irr(G_1(F))$ such that $\pi'|_{G(F)}=\pi$ and $\pi'|_{Z(G_1(F))}= z_{\varphi_1}\cdot\Id$. It remains to check that the $L$-parameter $\varphi_{\pi'}$ equals $\varphi_1$.

\smallskip

Since $\pi'|_{G(F)}=\pi$, we deduce from \cite[Theorem~IX.0.5(v)]{FS} that $\varphi_{\pi'}$ is also a lift $\varphi$. Consider torus $C:=G_1/G=Z(G_1)/Z(G)$. Then, by \cite[Corollary~5.13]{Ka}, $\varphi_{\pi'}$ differs from $\varphi_1$ by an element of
$\al\in H^1(W_F,\wh{C})$.

\smallskip

On the other hand, since $\pi'|_{Z(G_1(F))}=\pi_1|_{Z(G_1(F))}$, we conclude from
\cite[Theorem~IX.0.5(iii)]{FS} that the image of $\al$ in $H^1(W_F,\wh{Z(G_1)})$ is zero. Hence $\al=0$, because the projection
$\pr:Z(G_1)(F)\to C(F)$ is surjective (by \cite[Fact~5.5]{Ka}), and thus the induced map $\pr^*:H^1(W_F,\wh{C})\to H^1(W_F,\wh{Z(G_1)})$ is injective.
\end{Emp}

\begin{Emp}
{\bf Proof of part~(b).} The assertion follows by a small modification of the argument of \cite[Theorem~IX.6.1]{FS}: By the construction (see Section~\re{main}) and observation of Section~\re{compat}(b), it remains to show that for every homomorphism $\eta:G\to G_1$ of reductive groups over $F$ inducing an isomorphism of adjoint groups, the pullback $\eta^*:D(\Bun_{G_1})\to D(\Bun_{G})$ is compatible with 
with Hecke actions.
\end{Emp}

In other words, in the notation of \cite[Theorem~IX.6.1]{FS} the assertion follows from the following claim.

\begin{Cl} \label{C:FS}
For every finite set $I$, every representation $V\in \Rep_{\qlbar}((\wh{G}\ltimes Q)^I)$ and every $A\in D(\Bun_{G_1})$, we have a canonical isomorphism
\begin{equation} \label{Eq:FS}
T_V(\eta^*A)\simeq\eta^*T_{\wh{\eta}^*V}(A)\in D(\Bun_G\times(\on{Div}^1)^I).
\end{equation}
\end{Cl}
\begin{proof}
To prove the result, we repeat the argument of \cite[Theorem~IX.6.1]{FS} almost word-by-word. Consider commutative diagram
\begin{equation*} 
\begin{CD}
\Bun_{G} @<<h_l< \on{Hck}^I_G @>>h'_r> \on{Hck}^I_{G_1}\times_{\Bun_{G_1}}\Bun_{G} @>>h''_r> \Bun_G\times(\on{Div}^1)^I \\
@V\eta VV                  @V\eta_H VV               @V\wt{\eta} VV                        @V\eta VV \\
\Bun_{G_1} @<h_{l1}<<\on{Hck}^I_{G_1} @= \on{Hck}^I_{G_1} @>h_{r1}>> \Bun_{G_1}\times(\on{Div}^1)^I.
\end{CD}
\end{equation*}
Then, as in the proof of \cite[Theorem~IX.6.1]{FS}, we have the following isomorphisms
\[
T_V(\eta^*A)= h''_{r\natural}h'_{r\natural}(h_l^*\eta^*A\sotimes S_V)\simeq
h''_{r\natural}h'_{r\natural}(\eta_H^* h_{l1}^*A\sotimes S_V){\simeq} \]\[
{\simeq} h''_{r\natural}h'_{r\natural}(h'^*_{r\natural}\wt{\eta}^*h_{l1}^*A\sotimes S_V)\overset{(1)}{\simeq}
h''_{r\natural}
(\wt{\eta}^*h_{l1}^*A\sotimes h'_{r\natural}(S_V))\overset{(2)}{\simeq}
h''_{r\natural}
(\wt{\eta}^*h_{l1}^*A\sotimes \wt{\eta}^*(S_{\wh{\eta}^*V})){\simeq}
\]\[
{\simeq} h''_{r\natural}
\wt{\eta}^*(h_{l1}^*A\sotimes S_{\wh{\eta}^*V})\overset{(3)}{\simeq} \eta^*h_{1r\natural}(h_{l1}^*A\sotimes S_{\wh{\eta}^*V})=
\eta^*T_{\wh{\eta}^*V}(A),
\]
where isomorphism $(1)$ follows from the projection formula, isomorphism (3) from the base change isomorphism, and isomorphism (2) from isomorphism $h'_{r\natural}(S_V)){\simeq} \wt{\eta}^*(S_{\wh{\eta}^*V})$ obtained from the compatibility of the geometric Satake equivalence with $\eta$.
\end{proof}

\section{Elliptic virtual characters, and proof of Propositions~\ref{P:ell}} \label{ell}
In this and the next section, we choose a field isomorphism $\qlbar\isom\B{C}$ and consider representations to be $\B{C}$-valued rather than $\qlbar$-valued.

\begin{Emp} \label{E:ellvirt}
{\bf Elliptic virtual characters.}
(a) Let $\Theta_{\on{ell}}(G(F))\subset \Theta(G(F))$ be the linear space of characters $\chi\otimes\mu$, where $\chi\in \Theta(G(F))$  is an elliptic tempered character in the sense of Arthur \cite{Ar1}, and $\mu:G(F)\to \B{R}_{>0}$ is a homomorphism.

\smallskip

(b) By definition, a character $\chi_{\pi}$ of every supercuspidal representation $\pi$ belongs to $\Theta_{\on{ell}}(G(F))$. More generally, the same is true for every $\pi\in\Irr(G(F))$, whose restriction to $G^{\on{der}}(F)$ is square-integrable.

\smallskip

(c) It follows from \cite[Section~6]{Ar1} that every nonzero $\chi\in \Theta_{\on{ell}}(G(F))$ has a nonzero restriction to
$G(F)^{\on{ell}}$.

\smallskip

(d) We denote by $\Psi_{\on{ell}}(G)$ the set of equivalences classes of elliptic endoscopic data for $G$.

\smallskip

(e) For every $\C{E}\in \Psi_{\on{ell}}(G)$, we denote by $\Theta_{\C{E}-\st}(G(F))\subseteq \Theta(G(F))$ the subspace of $\C{E}$-stable generalized functions and set
\[
\Theta_{\C{E}}(G(F)):=\Theta_{\on{ell}}(G(F))\cap \Theta_{\C{E}-\st}(G(F))\subseteq\Theta(G(F)).
\]
\end{Emp}

Using observation of Section~\re{ellvirt}(b) above, \rp{ell} is a particular case of part~(b) of the following assertion: 

\begin{Prop} \label{P:end}
(a) We have a decomposition
\[
\Theta_{\on{ell}}(G(F))=\bigoplus_{\C{E}\in\Psi_{\on{ell}}(G)}\Theta_{\C{E}}(G(F)).
\]

\smallskip

(b) Let $\C{E}$ be an elliptic endoscopic triple of $G$, and let $\chi\in\Theta_{\on{ell}}(G(F))$ be such that the restriction $\chi^{\on{ell}}=\chi|_{G(F)^{\on{ell}}}\in C^{\infty}_{\inv}(G(F)^{\on{ell}})$ is $\C{E}$-stable. Then $\chi$ is $\C{E}$-stable.
\end{Prop}

\begin{proof}
(a) First we claim that every $\chi\in \Theta_{\on{ell}}(G(F))$ can be written as a finite sum $\chi=\sum_{\C{E}\in\Psi_{\on{ell}}(G)}\chi_{\C{E}}$ such that $\chi_{\C{E}}\in\Theta_{\on{ell}}(G(F))$ is $\C{E}$-stable for every $\C{E}\in\Psi_{\on{ell}}(G)$. If $\chi\in \Theta(G(F))$  is an elliptic tempered character in the sense of Arthur, this follows from a deep theorem of Arthur \cite[Theorem~XI.4]{MW}, from which the general case follows by twisting.

\smallskip

It remains to show that the subspaces  $\{\Theta_{\C{E}}(G(F))\}_{\C{E}\in\Psi_{\on{ell}}(G)}\subseteq \Theta(G(F))$ are linearly independent.
Using Section~\re{ellvirt}(c), the assertion follows from \rl{linend} below.

\smallskip

(b) By  part (a), invariant generalized function $\chi$ can be written as a finite sum
$\chi=\sum_{\C{E'}\in\Psi_{\on{ell}}(G)}\chi_{\C{E}'}$, where each $\chi_{\C{E}'}\in \Theta_{\on{ell}}(G(F))$ is $\C{E}'$-stable.
It remains to show that $\chi_{\C{E}'}=0$ for every $\C{E}'\not\simeq\C{E}$.

\smallskip

Restricting to $G(F)^{\on{ell}}$, we get an identity
%
$\chi^{\on{ell}}=\sum_{\C{E'}}\chi_{\C{E}'}^{\on{ell}}\in C^{\infty}_{\inv}(G(F)^{\on{ell}})$,
%
and each $\chi_{\C{E}'}^{\on{ell}}$ is $\C{E}'$-stable. Since $\chi^{\on{ell}}$ is $\C{E}$-stable by assumption, we conclude from \rl{linend} below that $\chi_{\C{E}'}^{\on{ell}}=0$ for every $\C{E}'\not\simeq\C{E}$. Then $\chi_{\C{E}'}=0$ by the observation of Section~\re{ellvirt}(c).
\end{proof}

\begin{Lem} \label{L:linend}
The subspaces $\{C^{\infty}_{\C{E}-\st}(G(F)^{\on{ell}})\}_{\C{E}\in\Psi_{\on{ell}}(G)}\subseteq C^{\infty}_{\inv}(G(F)^{\on{ell}})$ are linearly independent.
\end{Lem}

\begin{proof}
Unwinding definitions (see Sections~\re{estable}(b),(c)) we have to show that for every $\gm\in G(F)^{\on{ell}}$, the non-zero subspaces
\[
\{\Span_{\qlbar}\on{Fun}([\on{Orb}^{\st}_{\gm}])_{\xi,(\gm,\xi)\in\C{E}}\}_{\C{E}}\subseteq \on{Fun}([\on{Orb}^{\st}_{\gm}])
\]
are linearly independent. Since subspaces $\on{Fun}([\on{Orb}^{\st}_{\gm}])_{\xi}$ are linearly independent,
this follows from Remark~\re{stable}(d).
\end{proof}

\appendix
\section{Endoscopic decomposition over $\B{C}$} \label{end}

\begin{Emp} \label{E:cons}
{\bf Construction.} (a) Let $M\subseteq G$ be a Levi subgroup. Every endoscopic datum $\C{E}=(s,\wh{H},\rho)$ for $M$ gives rise to an endoscopic datum $\C{E}_G=(\wt{s},\wh{H},\rho)$ for $G$, unique up to equivalence:

\smallskip

Namely, the inclusion $M\hra G$ gives rise to a $\Gm$-invariant conjugacy class of inclusions $\wh{M}\hra\wh{G}$, hence of inclusions  $\wh{M}/Z(\wh{G})\hra\wh{G}^{\ad}$. Then, for every generic lift $\wt{s}\in \wh{M}/Z(\wh{G})\subseteq\wh{G}^{\ad}$ of $s\in\wh{M}^{\ad}=\wh{M}/Z(\wh{M})$, we have $Z_{\wh{G}}(\wt{s})^0=Z_{\wh{M}}(s)^0=\wh{H}$, and we define $\C{E}_G$ to be
$(\wt{s},\wh{H},\rho)$.

\smallskip

(b) Set $W(M):=N_{G}(M)/M$. Then for every $w\in W(M)$ we can form another endoscopic datum  $w(\C{E})$ for $M$, and,
by construction, endoscopic data $\C{E}_G$ and $w(\C{E})_G$ from part~(a) are equivalent.
\end{Emp}

\begin{Emp} \label{E:relevant}
{\bf Relevant endoscopic data.} (a) We say that an endoscopic datum $\C{E}$ is {\em relevant}, if there exists  $(\gm,\xi)\in\C{E}$ (see Section~\re{stable}(a)).

\smallskip

(b) Note that every endoscopic datum $\C{E}$ is relevant, if $G$ is quasi-split. Moreover, $\C{E}$ is relevant if and only if the space $\wh{C}_{\C{E}-\st}(G(F))$ is nonzero. Furthermore, this happens if and only if there exists a Levi subgroup $M\subseteq G$ and an elliptic endoscopic datum $\C{E}'$ for $M$ such that $\C{E}\simeq\C{E}'_G$. In this case, $M:=M_{\C{E}}$ is unique up to a conjugacy, and the equivalence class of $\C{E}'$ is unique up to a $W(M)$-conjugacy.
\end{Emp}

\begin{Emp} \label{E:notation}
{\bf Notation.}
(a) We denote by $\Psi(G)$ the set of equivalences classes of endoscopic data for $G$.

\smallskip

(b) For every $\C{E}\in\Psi(G)$, we denote by $\Theta_{\C{E}-\st}(G(F))\subseteq\Theta(G(F))$ the subspace of $\C{E}$-stable generalized functions.

\smallskip

(c) For every Levi subgroup $M\subseteq G$, we denote by
$\iota_{M}^G:\Theta(M(F))\to \Theta(G(F))$ the normalized parabolic induction, and set
\[
\Theta_{M-\on{ell}}(G(F)):=\iota_{M}^G(\Theta_{\on{ell}}(M(F)))\subseteq\Theta(G(F)).
\]

\smallskip

(d) For every endoscopic datum $\C{E}$ for $G$, we set
\[
\Theta_{\C{E}-\on{ell}}(G(F)):=\Theta_{M_{\C{E}}-\on{ell}}(G(F))\subseteq\Theta(G(F)),
\]
if $\C{E}$ is relevant and $M_{\C{E}}\subseteq G$ is the Levi subgroup defined in Section~\re{relevant}(b), and zero subspace, otherwise.
In particular, we have $\Theta_{\C{E}-\on{ell}}(G(F))=\Theta_{\on{ell}}(G(F))$, if $\C{E}\in\Psi_{\on{ell}}(G)$.

\smallskip

(e) For every endoscopic datum $\C{E}$ for $G$, we set
\[
\Theta_{\C{E}}(G(F)):=\Theta_{\C{E}-\on{ell}}(G(F))\cap \Theta_{\C{E}-\st}(G(F))\subseteq\Theta(G(F)).
\]
\end{Emp}










The following result, which  we learned from Waldspurger, is due to Arthur. For completeness, we will sketch its proof in
Section~\re{pfdec}.

\begin{Prop} \label{P:dec}
(a) For every Levi subgroup $M\subseteq G$, the map $\iota_{M}^G$ induces an isomorphism
\[
\iota_{M}^G:\Theta_{\on{ell}}(M(F))^{W(M)}\isom\Theta_{M-\on{ell}}(G(F)).
\]

\smallskip

(b) We have a decomposition
\[
\Theta(G(F))=\bigoplus_{M}\Theta_{M-\on{ell}}(G(F)),
\]
where $M$ runs over the set of conjugacy classes of Levi subgroups in $G$.
\end{Prop}

The following corollary of \rp{end} will be proven in Section~\re{pfdec1}.

\begin{Prop} \label{P:dec1}
For every Levi subgroup $M\subseteq G$, we have a decomposition
\[
\Theta_{M-\on{ell}}(G(F))=\bigoplus_{\C{E}\in\Psi_{\on{ell}}(M)/W(M)}\Theta_{\C{E}_G}(G(F)).
\]
\end{Prop}



We are ready to formulate and to prove the main result of this section.

\begin{Thm} \label{T:dec}
We have an endoscopic decomposition
\[
\Theta(G(F))=\bigoplus_{\C{E}\in\Psi(G)}\Theta_{\C{E}}(G(F)).
\]
\end{Thm}

\begin{proof}
Combining Propositions~\ref{P:dec}(b) and \ref{P:dec1}, we have a decomposition
\[
\Theta(G(F))=\bigoplus_M\bigoplus_{\C{E}\in\Psi_{\on{ell}}(M)/W(M)}\Theta_{\C{E}_G}(G(F))
\]
from which the assertion follows by the observations of Section~\re{relevant}(b).
\end{proof}

\begin{Emp}
{\bf Remark.} Notice that the decomposition of \rt{dec} is more refined than the one people usually consider.
For example, the stable part of $\Theta(G(F))$ corresponds to a direct sum  $\bigoplus_M\Theta_{\C{E}_{M,\on{triv}}}(G(F))$,
where $M$ run over the set of conjugacy classes of Levi subgroups $M$ of $G$.
\end{Emp}

We will need basic properties of the normalized parabolic induction, which we remind now.

\begin{Emp} \label{E:par}
{\bf Normalized parabolic induction.} Let $M\subseteq G$ be a Levi subgroup.

\smallskip

(a) The normalized parabolic induction $\iota_{M}^G:\Theta(M(F))\to \Theta(G(F))$
extends to a continuous linear map $\iota_{M}^G:\wh{C}_{\inv}(M(F))\to \wh{C}_{\inv}(G(F))$
(see, for example, \cite[Corollaries~6.3 and ~6.6(a)]{KV2}).

\smallskip

(b) For every $\gm\in G(F)$, we set $\Dt_{M,G}(\gm)=\det(\Ad_{\gm}^{-1}-\Id,\Lie G/\Lie H)$.
Then for every $\gm\in G^{\sr}(F)\cap M(F)^{\on{ell}}$ and $\chi\in \Theta(M(F))$, we have an equality
\[
\iota_{M}^G(\chi)(\gm)=|\Dt_{M,G}(\gm)|^{-1/2}\cdot \sum_{w\in W(M)}\chi(w(\gm)).
\]
For example, it can deduced from \cite[Corollary~6.4]{KV2} with $H=M$.

\smallskip

(c) As in Section~\re{stable}(e), for every $\gm\in G^{\sr}(F)\cap M(F)\subseteq M^{\sr}(F)$, we can consider
orbital integrals $O_{M,\gm}\in \wh{C}_{\inv}(M(F))$ and $O_{G,\gm}\in \wh{C}_{\inv}(G(F))$. Then, as it is shown
for example, in \cite[Corollary~6.11]{KV2}, we have an equality
\[
\iota_{M}^G(O_{M,\gm})=|\Dt_{M,G}(\gm)|^{-1/2}\cdot O_{G,\gm}.
\]

\smallskip

(d) For every $\gm\in G^{\sr}(F)\cap M(F)\subseteq M^{\sr}(F)$, we have equality $G_{\gm}=M_{\gm}$. Moreover, the
natural homomorphism $(\wh{G}_{\gm})^{\Gm}/Z(\wh{G})^{\Gm}\to(\wh{M}_{\gm})^{\Gm}/Z(\wh{M})^{\Gm}$ induces an isomorphism
$\pi_0((\wh{G}_{\gm})^{\Gm}/Z(\wh{G})^{\Gm})\isom\pi_0((\wh{M}_{\gm})^{\Gm}/Z(\wh{M})^{\Gm})$.

\smallskip

(e) By part~(d), for every $\ov{\xi}\in\pi_0((\wh{G}_{\gm})^{\Gm}/Z(\wh{G})^{\Gm})=\pi_0((\wh{M}_{\gm})^{\Gm}/Z(\wh{M})^{\Gm})$,
one can consider $\ov{\xi}$-orbital integrals $O^{\ov{\xi}}_{M,\gm}\in \wh{C}_{\inv}(M(F))$ and $O^{\ov{\xi}}_{G,\gm}\in \wh{C}_{\inv}(G(F))$.
Furthermore, it follows from part~(c) that we have an identity
\[
\iota_{M}^G(O^{\ov{\xi}}_{M,\gm})=|\Dt_{M,G}(\gm)|^{-1/2}\cdot O^{\ov{\xi}}_{G,\gm}.
\]
\end{Emp}

\begin{Emp} \label{E:rem}
{\bf Remarks.} Let $M\subseteq G$ be a Levi subgroup, $\gm\in G^{\sr}(F)\cap M(F)\subseteq M^{\sr}(F)$, and let $\C{E}\in\Psi(M)$.

\smallskip

(a) By the construction of Section~\re{cons}, for every  $\xi\in (\wh{G}_{\gm})^{\Gm}/Z(\wh{M})^{\Gm}$ such that $(\gm,\xi)\in \C{E}$
there exists a lift $\wt{\xi}\in (\wh{G}_{\gm})^{\Gm}/Z(\wh{G})^{\Gm}$ of $\xi$ such that $(\gm,\wt{\xi})\in \C{E}_G$.

\smallskip

(b) Conversely, assume that $\gm\in G^{\sr}(F)\cap M(F)^{\on{ell}}$. Then for every  $\xi\in (\wh{G}_{\gm})^{\Gm}/Z(\wh{G})^{\Gm}$ such that $(\gm,\xi)\in \C{E}_G$ there exists $w\in W(M)$ such that the image   $\ov{\xi}\in (\wh{G}_{\gm})^{\Gm}/Z(\wh{M})^{\Gm}$ satisfies
$(\gm,\xi)\in w(\C{E})$.
\end{Emp}

\begin{Prop} \label{P:parind}
For every Levi subgroup $M\subseteq G$ and $\C{E}\in\Psi(M)$ we have an inclusion
\[
\iota_M^G(\wh{C}_{\C{E}-\st}(M(F)))\subseteq \wh{C}_{\C{E}_G-\st}(G(F)).
\]
\end{Prop}
\begin{proof}
The proof is a straightforward generalization of the argument of \cite[Corollary~6.13]{KV2}, where the stable case is proven.
Note that $\wh{C}_{\C{E}-\st}(M(F))$ is the closure of the linear span of $\{O_{M,\gm}^{\ov{\xi}}\}_{(\gm,\xi)}$, where $\gm\in G^{\sr}(F)\cap M(F)\subseteq M^{\sr}(F)$ and $\xi\in (\wh{G}_{\gm})^{\Gm}/Z(\wh{M})^{\Gm}$ such that $(\gm,\xi)\in \C{E}$.
Since $\iota_M^G$ is continuous, it suffices to show that for every such $(\gm,\xi)$ we have
$\iota_M^G(O_{M,\gm}^{\ov{\xi}})\in \wh{C}_{\C{E}_G-\st}(G(F))$. Then, by the identity of Section~\re{par}(e), we have to show that
$O_{G,\gm}^{\ov{\xi}}\in \wh{C}_{\C{E}_G-\st}(G(F))$. The latter assertion follows
by the observation of Section~\re{rem}(a).
\end{proof}

\begin{Emp} \label{E:pfdec}
\begin{proof}[Proof of \rp{dec}]
We have to show that every $\chi\in\Theta(G(F))$ can be written uniquely as a finite sum $\chi=\sum_{M}\iota_M^G(\chi_M)$ with
$\chi_M\in \Theta_{\on{ell}}(M(F))^{W(M)}$ for all $M$.

\smallskip

When $\chi$ belongs to the span of characters of tempered representations of $G(F)$, the existence of the $\chi_M$'s is shown
for example in \cite[Proposition~I, 2.12]{MW2}. The existence of the $\chi_M$'s in general now follows from the fact that character of each
$\pi\in \Irr(G(F))$ can be written as a linear combination of characters of
representations of the form $\iota^G_M(\pi_M\otimes\mu_M)$, where $\pi_M\in\Irr(M(F))$ is tempered, and $\mu_M$ is a character
$M(F)\to\B{R}_{>0}$.  

\smallskip

To show the uniqueness of the $\chi_M$'s, we have to show that for every collection $\chi_M\in \Theta_{\on{ell}}(M(F))^{W(M)}$ such that
$\sum_{M}\iota_M^G(\chi_M)=0$ we have $\chi_M=0$ for all $M$.

\smallskip

Assuming this is not the case, we choose a Levi subgroup $L\subseteq G$
such that $\chi_{L}\neq 0$ and such that the dimension of $Z(L)$ is minimal. Since $\chi_{L}\in \Theta_{\on{ell}}(L(F))^{W(L)}$, we conclude that the
restriction $\chi^{\on{ell}}_L\in C^{\infty}_{\inv}(L(F)^{\on{ell}})$ is non zero (by Section~\re{ellvirt}(c)). Thus, there exists an element $\gm\in L(F)^{\on{ell}}\cap G(F)^{\sr}$ such that $\chi_L(\gm)\neq 0$. Then, by Section~\re{par}(b), we conclude that $\iota_L^G(\chi_L)(\gm)\neq 0$.

\smallskip

Since $\sum_{M}\iota_M^G(\chi_M)=0$, there exists at least one Levi subgroup $M$, not conjugate to $L$, such that $\iota_M^G(\chi_M)(\gm)\neq 0$.
Then, by the formula for the induced representation, $M$ must contain a conjugate of $\gm$. Since $\gm\in L(F)^{\on{ell}}$,
$M$ must then contain a conjugate of $L$. Hence $\chi_M\neq 0$, and $\dim Z(M)<\dim Z(L)$, contradicting our assumption on $L$.
\end{proof}
\end{Emp}

\begin{Emp} \label{E:pfdec1}
\begin{proof}[Proof of \rp{dec1}] By a combination of Propositions~\ref{P:end} and ~\ref{P:parind}, every $\chi\in\Theta_{M-\on{ell}}(G(F))$
can be written as a finite sum $\chi=\sum_{\C{E}\in\Psi_{\on{ell}}(M)}\chi_{\C{E}}$ such that
$\chi_{\C{E}}\in \Theta_{\C{E}_G}(G(F))$ for all $\C{E}$.
\smallskip

It remains to show that the subspaces $\{\Theta_{\C{E}_G}(G(F))\}_{\C{E}}\subseteq \Theta_{M-\on{ell}}(G(F))$ are linearly independent.
By a combination of \rp{end}(a) and Section~\re{ellvirt}(c), the restriction map $\Theta_{M-\on{ell}}(G(F))\to C^{\infty}(M(F)\cap G^{\sr}(F))$ is injective. Then, assertion follows by the same argument as  \rp{dec}(a) using Section~\re{rem}(b).
\end{proof}
\end{Emp}

\section{The case of an arbitrary field of coefficients} \label{S:arb}

Let $K$ be an arbitrary algebraically closed field of characteristic zero. Combining results of the previous section and a recent work \cite{KSV}, in this section we are going to construct Arthur's space of virtual elliptic characters and endoscopic decomposition over $K$.

\begin{Emp} \label{E:setup}
{\bf General observations.} 
 (a) We denote by $\Irr_K(G(F))$ the set of isomorphism classes of smooth irreducible representations $G(F)$ over $K$ and by $\wh{C}_{\inv}(G(F),K)$ be the space of $K$-valued invariant generalized functions on $G(F)$.

\smallskip

(b) We claim that every $\pi\in \Irr_K(G(F))$ is admissible, and hence we can consider its character $\chi_{\pi}\in \wh{C}_{\inv}(G(F),K)$.
Indeed, the argument of \cite[Theorem~3.4]{KSV}, where the corresponding assertion was proven $K=\ov{\B{Q}}$, applies word-by-word. 
We denote by $\Theta_K(G(F))\subseteq \wh{C}_{\inv}(G(F),K)$ the linear span of $\{\chi_{\pi}\}_{\pi\in \Irr_K(G(F))}$.

\smallskip

(c) Arguing as in \cite[Proposition~3.5]{KSV} word-by-word, it follows from part~(b) that for every algebraically closed subfield $L\subseteq K$ and $\pi\in \Irr_{L}(G(F))$ the representation $\pi_K:=\pi\otimes_{L}K$ is irreducible and that $\pi'_K\not\simeq \pi''_K$ if $\pi'\not\simeq \pi''$. Therefore the map $\pi\mapsto \pi_K$ induces an injective map $\Theta_{L}(G(F))\hra\Theta_K(G(F))$.

\smallskip

(d) Since the Hecke algebra of $G$ is of countable dimension, for every representation $\pi\in \Irr_{K}(G(F))$ there exists a countable subfield
$L\subseteq K$ such that $\pi$ is obtained by a base change from certain representation in $\Irr_{L}(G(F))$. In particular, the space
$\Theta_K(G(F))$ can be written as a filtered union $\bigcup_{L\subseteq K} \Theta_{L}(G(F))$, taken over countable algebraically closed subfields of $K$.





\end{Emp}

\begin{Emp} \label{E:arthur}
{\bf Elliptic virtual characters.} For every algebraically closed field $K$ of characteristic zero we define
a subspace $\Theta_{\on{ell},K}(G(F))\subseteq \Theta_{K}(G(F))$ as follows:

\smallskip

(a) If $K=\B{C}$, we denote $\Theta_{\on{ell},\B{C}}(G(F))\subseteq \Theta_{\B{C}}(G(F))$ the subspace
$\Theta_{\on{ell}}(G(F))\subseteq \Theta(G(F))$ from Section~\re{ellvirt}(a).

\smallskip

(b) Note that the Galois group $\Gal(\B{C}/\B{Q})$ acts on $\Theta_{\B{C}}(G(F))$ by the action
$\si(\sum_i a_i\chi_{\pi_i})=\sum_i \si(a_i)\chi_{\si(\pi_i)}$. Moreover, using remark after \cite[Lemma~5.8]{KSV} it follows from \cite[Theorem~1.3.(2)]{KSV} that
the subspace $\Theta_{\on{ell},\B{C}}(G(F))\subseteq \Theta_{\B{C}}(G(F))$ is $\Gal(\B{C}/\B{Q})$-invariant.

\smallskip

(c) Assume first that $K$ is countable. Then there exists a field embedding $\iota:K\hra \B{C}$ and we define
$\Theta_{\on{ell},K}(G(F))\subseteq \Theta_{K}(G(F))$ to be $\Theta_{K}(G(F))\cap \Theta_{\on{ell},\B{C}}(G(F))$.
Notice that every two field embeddings $\iota:K\hra \B{C}$ differ by an automorphism of $\B{C}$. Thus it follows from part~(b) that
the subspace $\Theta_{\on{ell},K}(G(F))$ is independent of $\iota$. Moreover, for every algebraically closed subfield $L\subseteq K$,
we have an equality $\Theta_{\on{ell},L}(G(F))=\Theta_{\on{ell},K}(G(F))\cap \Theta_{L}(G(F))$.

\smallskip

(d) For an arbitrary $K$, we denote by $\Theta_{\on{ell},K}(G(F))\subseteq \Theta_{K}(G(F))$ the union $\bigcup_{L\subseteq  K}\Theta_{\on{ell},L}(G(F))$, taken over all countable algebraically closed subfields $L$. Notice that the collection of such subfields is filtered, so this union is a
vector subspace.
\end{Emp}

\begin{Emp} \label{E:enddec}
{\bf Endoscopic decomposition.} Using subspace  $\Theta_{\on{ell},K}(G(F))\subseteq \Theta_{K}(G(F))$ defined in Section~\re{arthur}(d),
the whole construction of Section~\re{notation} can be carried out:

\smallskip

(a) The dual group $\wh{G}$ is defined over $\ov{\B{Q}}$, thus over $K$. Next, the definition of endoscopic data is purely algebraic, so it makes sense over $K$. Moreover, the set $\Psi_K(G)$ of equivalences classes of endoscopic data for $G$ over $K$ does not depend of $K$ and
so will be denoted by $\Psi(G)$.

\smallskip

(b) Repeating the definition of Section~\re{estable} word-by-word, to every $\C{E}\in\Psi(G)$ we associate a subspace
$\Theta_{\C{E}-\st,K}(G(F))\subseteq\Theta_K(G(F))$ of $\C{E}$-stable generalized functions.

\smallskip

(c) Finally, repeating definitions of Section~\re{notation}(c)-(e), to every $\C{E}\in\Psi(G)$ we associate a subspace
\[
\Theta_{\C{E},K}(G(F))\subseteq\Theta_K(G(F)).
\]
\end{Emp}

\begin{Thm} \label{T:dec'}
We have an endoscopic decomposition
\[
\Theta_K(G(F))=\bigoplus_{\C{E}\in\Psi(G)}\Theta_{\C{E},K}(G(F)).
\]
\end{Thm}

\begin{proof}
Using observation of Section~\re{setup}(d) and the constructions of Sections~\re{arthur} and \re{enddec}, we can assume that $K$ is a countable subfield of $\B{C}$. Then, by \rt{dec}, every $\chi\in \Theta_K(G(F))$ can be written uniquely as a finite sum
\begin{equation*} \label{Eq:dec'}
\chi=\sum_{\C{E}\in\Psi(G)}\chi_{\C{E}}
\end{equation*}
such that $\chi_{\C{E}}\in \Theta_{\C{E},\B{C}}(G(F))$ for every $\C{E}\in\Psi(G)$, and we need to show that for every $\C{E}\in\Psi(G)$ we have
$\chi_{\C{E}}\in \Theta_{\C{E},K}(G(F))$.

For this it suffices to show that each subspace $\Theta_{\C{E},\B{C}}(G(F))\subseteq\Theta_{\B{C}}(G(F))$ is $\Gal(\B{C}/\B{Q})$-invariant, and
$\Theta_{\C{E},K}(G(F))=\Theta_{\C{E},\B{C}}(G(F))^{\Gal(\B{C}/K)}$.

\smallskip

Since $\Theta_{\C{E},\bullet}(G(F))=\Theta_{\C{E}-\st,\bullet}(G(F))\cap\Theta_{M_{\C{E}}-\on{ell},\bullet}(G(F))$, it suffices to show similar assertions for $\Theta_{\C{E}-\st,\bullet}(G(F))$ and $\Theta_{M_{\C{E}}-\on{ell},\bullet}(G(F))$. Next, by \rp{dec}(a), the assertion for  $\Theta_{M_{\C{E}}-\on{ell},\bullet}(G(F))$ follows from that for  $\Theta_{\on{ell},\bullet}(M(F))$. Finally, the assertions for
 $\Theta_{\C{E}-\st,\bullet}(G(F))$ and $\Theta_{\on{ell},\bullet}(G(F))$ follows from
\rl{gal}(b) below.
\end{proof}


\begin{Lem} \label{L:gal}
Let $K\subseteq\B{C}$ be an algebraically closed subfield.
\smallskip

(a) For every $\pi\in \Irr_{\B{C}}(G(F))$ not defined over $K$, the $\Gal(\B{C}/K)$-orbit of $\pi$ is infinite.

\smallskip

(b) We have an equality  $\Theta_{K}(G(F))=\Theta_{\B{C}}(G(F))^{\Gal(\B{C}/K)}$.
\end{Lem}

\begin{proof}
(a) Our argument is a small modification of \cite[Theorem~3.10]{KSV}, where a particular case is proven.

\smallskip

Assume first that $\pi$ is cuspidal. Then, by \cite[Lemma~3.8]{KSV}, there exists an unramified character $\theta:G(F)\to\B{C}^{\times}$ such that the central character $cc(\pi')$ of $\pi':=\pi\otimes\theta^{-1}$ has values in $\ov{\B{Q}}^{\times}$. Then, \cite[Proposition~3.7.(2)]{KSV} says that $\pi'$ is defined over $\ov{\B{Q}}$. Since $\pi\in \Irr_{\B{C}}(G(F))$ is not defined over $K$, this implies that $\theta:G(F)\to\B{C}^{\times}$ has values not in $K^{\times}$. Then the $\Gal(\B{C}/K)$-orbit of  $\theta$ is infinite, hence the $\Gal(\B{C}/K)$-orbit $\pi=\pi'\otimes\theta$ is infinite.

\smallskip

For a general $\pi$, choose a parabolic subgroup $P\subseteq G$ such that the Jacquet functor $r_P^G(\pi)$ is cuspidal and non-zero. Then for every irreducible quotient $\tau$ of $r_P^G(\pi)$, $\pi$ is an irreducible subrepresentation of the parabolic induction $i_P^G(\tau)$. Then, repeating the argument of \cite[Lemma~3.6]{KSV}, we see that $\tau$ is not defined over $K$. Then, by the proven above, the $\Gal(\B{C}/K)$-orbit of $\tau$ is infinite.
But since $r_P^G(\pi)$ is of finite length, we conclude that the $\Gal(\B{C}/K)$-orbit of $\pi$ is infinite as well.

\smallskip

(b) Let $\chi=\sum_{i=1}^n a_i\chi_{\pi_i}\in \Theta_{\B{C}}(G(F))^{\Gal(\B{C}/K)}$, where $\pi_1,\ldots, \pi_n\in \Irr_{\B{C}}(G(F)$ are pairwise non-isomorphic representations and $a_1,\ldots,a_n\in \B{C}$. We have to show that $\pi_i$ is defined over $K$ and $a_i\in K$ for all $i$.
For every $\si\in\Gal(\B{C}/K)$, we have an identity
\begin{equation} \label{Eq:galconj}
\sum_{i=1}^n \si(a_i)\chi_{\si(\pi_i)}={\si}(\chi)=\chi=\sum_{i=1}^n a_i\chi_{\pi_i}.
\end{equation}
Finally, combining part~(a), identity \form{galconj} and linear independence of characters, we conclude that each $\pi_i$ is defined over $K$ and each  $a_i$ belongs to $K$. 
\end{proof}


\begin{thebibliography}{99}
\bibitem[Ar1]{Ar1}
J. Arthur, {\em On elliptic tempered characters}, Acta. Math. {\bf 171} (1993), 73--138.

\bibitem[Ar2]{Ar2}
J.~Arthur, {\em On local character relations}, Sel. Math., New Ser., {\bf 2} (1996), 501--579.

\bibitem[BV]{BV}
R.~Bezrukavnikov and Y.~Varshavsky, {\em Affine Springer fibers
and depth zero $L$-packets}, preprint, arXiv:2104.13123, 2021.

\bibitem[Bo]{Bo}
M.~Borovoi, {\em Abelian Galois cohomology of reductive groups}, Mem. Amer. Math. Soc. {\bf 132}, 1998.

\bibitem[FS]{FS}
L.~Fargues and P.~Scholze, {\em Geometrization of the local
Langlands correspondence}, preprint, arXiv:2102.13459, 2021.

\bibitem[Fu]{Fu}
Chenji Fu, {\em Stability of elliptic Fargues--Scholze $L$-packets}, preprint, arxiv:2501.0065, 2025.

\bibitem[GR]{GR}
D. Gaitsgory and N. Rozenblyum, {\em A study in derived algebraic geometry, Vol. 1: Correspondences and Duality},
Mathematical surveys and monographs {\bf 221} (2017), AMS, Providence, RI.

\bibitem[Ha]{Ha}
D.~Hansen, {\em Beijing notes on the categorical local Langlands conjecture},
preprint, arXiv:2310.04533, 2023.

\bibitem[HJ]{HJ}
D.~Hansen and C.~Johansson, {\em A note on the cohomology of moduli spaces of local shtukas},
preprint, arXiv:2404.04083, 2024.

\bibitem[HKW]{HKW}
D.~Hansen, T.~Kaletha, and Jared Weinstein, {\em On the Kottwitz conjecture for local shtuka spaces},
Forum Math. Pi, {\bf 10:79}, 2022. Id/No e13.

\bibitem[Ka]{Ka}
T.~Kaletha, {\em Rigid inner forms vs isocrystals}, J. Eur. Math. Soc. (JEMS) {\bf 20} (2018), no. 1, 61--101.

\bibitem[KV1]{KV}
D.~Kazhdan and Y.~Varshavsky, {\em Endoscopic decomposition of certain depth zero representations},
in {\em Studies in Lie theory}, 223--301, Progr. Math. {\bf 243}, Birkh\"auser Boston, MA, 2006.

\bibitem[KV2]{KV2}
D.~Kazhdan and Y.~Varshavsky, {\em Geometric approach to parabolic induction}, Selecta Math. (N.S.), {\bf 22(4)} (2016), 2243--2269.

\bibitem[Ko]{Ko}
R. Kottwitz, {\em Stable trace formula: cuspidal tempered terms}, Duke Math. J. {\bf 51} (1984),
611--650.

\bibitem[La]{La}
R.~P.~Langlands, {\em Les d\'ebuts d’une formule des traces stable}, Publ. Math. Univ. Paris VII
{\bf 13}, Paris, 1983.

\bibitem[Lu]{Lu}
J.~Lurie, {\em Higher algebra}, available at
https://www.math.ias.edu/~lurie/papers/HA.pdf.

\bibitem[MW1]{MW}
C.~Moeglin and J.-L.~Waldspurger, {\em Stabilisation de la
formule des traces tordue}, Vol. 2, Prog. Math. {\bf 317},
Birkh\"auser/Springer, Basel, 2016.

\bibitem[MW2]{MW2}
C.~Moeglin and J.-L.~Waldspurger, {\em La Formule des Traces Locale Tordue}, Memoirs of the AMS {\bf 1198}, 2018.

\bibitem[Sc]{Sc}
P.~Scholze, {\em Geometrization of the local Langlands correspondence, motivically}, preprint,
arXiv:2501.07944, 2025. 

\bibitem[KSV]{KSV}
D.~Kazhdan, M.~Solleveld and Y.~Varshavsky,
{\em Rationality properties of complex representations of reductive $p$-adic groups}, preprint, arXiv:2510.24201, 2025. 


\bibitem[Zou]{Zou}
Konrad Zou, {\em The categorical form of Fargues' conjecture for tori}, preprint, arXiv:2202.13238, 2022.
\end{thebibliography}
\end{document}